\documentclass{amsart}

\usepackage{amssymb}
\usepackage{mathtools}
\usepackage{stmaryrd}

\numberwithin{equation}{section}

\newtheorem{theo}{Theorem} 
\newtheorem{mainprop}[theo]{Proposition} 

\newtheorem{lemma}{Lemma}[section]
\newtheorem{prop}[lemma]{Proposition}
\newtheorem{corol}[lemma]{Corollary}

\newtheorem{claim}[lemma]{Claim}

\theoremstyle{remark}
\newtheorem{remark}[lemma]{Remark}

\newtheorem{step}{Step}

\theoremstyle{definition}
\newtheorem{defi}[lemma]{Definition}

\newcommand{\hdot}{\dot{H}^1}

\newcommand{\NN}{\mathbb{N}}
\newcommand{\RR}{\mathbb{R}}

\newcommand{\eps}{\varepsilon}

\newcommand{\tu}{\tilde{u}}

\newcommand{\EEE}{\mathcal{E}}
\newcommand{\HHH}{\mathcal{H}}

\newcommand{\LLL}{\mathcal{L}}

\newcommand{\PPP}{\mathcal{P}}

\newcommand{\indic}{1\!\!1}

\newcommand{\ent}[1]{\left\lfloor #1 \right\rfloor} 

\DeclareMathOperator{\vect}{span}

\DeclareMathOperator{\rad}{rad}
\title[Nonradiative solutions of critical wave equations]{Decay estimates for nonradiative solutions of the energy-critical focusing wave equation}

\author[T.~Duyckaerts]{Thomas Duyckaerts$^1$}
\author[C.~Kenig]{Carlos Kenig$^2$}
\author[F.~Merle]{Frank Merle$^3$}
\thanks{$^1$LAGA (UMR 7539), Universit\'e Paris 13, Sorbonne Paris Cit\'e, and Institut Universitaire de France}
\thanks{$^2$University of Chicago. Partially supported by NSF Grants DMS-14363746 and DMS-1800082}
\thanks{$^3$Cergy-Pontoise (UMR 8088), Institut des Hautes \'Etudes Scientifiques}


\begin{document}

\begin{abstract}
 Consider the energy-critical focusing wave equation in space dimension $N\geq 3$. The equation has a nonzero radial stationary solution $W$, which is unique up to scaling and sign change. It is conjectured (soliton resolution) that any radial, bounded in the energy norm  solution of the equation behaves asymptotically as a sum of modulated $W$s, decoupled by the scaling, and a radiation term. A \emph{nonradiative} solution of the equation is by definition a solution whose energy in the exterior $\{|x|>|t|\}$ of the wave cone vanishes asymptotically as $t\to +\infty$ and $t\to -\infty$. In our previous work \cite{DuKeMe13}, we have proved that the only radial nonradiative solutions of the equation in three space dimensions are, up to scaling, $0$ and $\pm W$. This was crucial in the proof of soliton resolution in \cite{DuKeMe13}. 

In this paper, we prove that the initial data of a radial nonradiative solution in odd space dimension have a prescribed asymptotic behaviour as $r\to \infty$. We will use this property for the proof of soliton resolution, for radial data, in all odd space dimensions. The proof uses the characterization of nonradiative solutions of the linear wave equation in odd space dimensions obtained by Lawrie, Liu, Schlag and the second author in \cite{KeLaLiSc15}. 

We also study the propagation of the support of nonzero radial solutions with compactly supported initial data, and prove that these solutions cannot be nonradiative.
\end{abstract}
\maketitle
\tableofcontents

\section{Introduction}
Consider the wave equation on $\RR^N$, $N\geq 3$, with an energy-critical focusing nonlinearity:
\begin{equation}
 \label{NLW}
 \partial_t^2u-\Delta u=|u|^{\frac{4}{N-2}}u,
\end{equation} 
and initial data
\begin{equation}
 \label{ID}
 \vec{u}_{\restriction t=0}=(u_0,u_1)\in \HHH,
\end{equation} 
where $\vec{u}:=(u,\partial_t u)$, $\HHH:=\hdot(\RR^N)\times L^2(\RR^N)$, where $\hdot(\RR^N)$ is the usual homogeneous Sobolev space. We will mainly consider spherically symmetric solutions, and denote by $\HHH_{\rad}$ the subspace of radial elements of $\HHH$

The equation is locally well-posed in $\HHH$ (see e.g. \cite{GiSoVe92,LiSo95,KeMe08,BuCzLiPaZh13}). It has the following scaling invariance: if $u$ is a solution of \eqref{NLW} and $\lambda>0$, then 
$$\lambda^{\frac{N}{2}-1} u\left( \lambda t,\lambda x\right)$$
 is also a solution. The energy:
$$E(\vec{u}(t))=\frac{1}{2}\int_{\RR^N} |\nabla_{t,x}u(t,x)|^2\,dx-\frac{N-2}{2N}\int_{\RR^N} |u(t,x)|^{\frac{2N}{N-2}}\,dx,$$
where $\nabla_{t,x} u=(\partial_t,\partial_{x_1}u,\ldots,\partial_{x_N}u)$
is conserved, and invariant by the scaling of the equation

%
%
The equation \eqref{NLW} has stationary solutions. Indeed, the set of radial stationary solutions is exactly 
$$\{0\}\cup\left\{\pm\lambda^{\frac{N}{2}-1} W(\lambda x),\lambda>0\right\}, $$
where 
$$W(x)=\left( 1+\frac{|x|^2}{N(N-2)} \right)^{\frac{N}{2}-1}.$$
There also exist type I blow-up solutions, i.e. solutions whose maximal time of existence $T_+$ is finite and such that 
$$\lim_{t\to T_+}\|\vec{u}(t)\|_{\HHH}=+\infty.$$
The \emph{soliton resolution conjecture} predicts that every solution $\vec{u}$ of \eqref{NLW}, that is not a type I blow-up solution, decomposes asympotically as $t\to T_+$, up to a term which is negligible in $\HHH$, as a finite sum of stationary solutions, decoupled by time dependent scaling parameters, 
and a radiation term. The radiation term is a solution of the linear wave equation if $T_+=+\infty$ or a fixed element of $\HHH$ if $T_+<\infty$. 

This conjecture was proved by the authors in \cite{DuKeMe13} in space dimension $3$ in the radial setting. 
A key step of the proof is the
following rigidity theorem (see Propositions 2.1 and 2.2 in \cite{DuKeMe13}): let $u$ be a radial solution of \eqref{NLW} with $N=3$ which is global (to fix ideas), and such that
\begin{equation}
 \label{nonradi}
 \sum_{\pm}\lim_{t\to \pm\infty}\int_{|x|>|t|}|\nabla_{t,x}u(t,x)|^2\,dx=0
\end{equation} 
(we will call such solutions \emph{nonradiative}),
then $u$ is stationary, i.e. $u(t,x)=0$ or $u(t,x)=\pm \lambda^{\frac{N}{2}-1}W(\lambda x)$. 

It is essentially proved in \cite[Lemma 3.11]{DuKeMe13} that if $u$ is a solution of \eqref{NLW} with maximal time of existence $T_+$ and $t_n\to T_+$ is such that $\vec{u}(t_n)$ is bounded in $\HHH$, then $u$ can be expanded, in a neighborhood of $t_n$, as a sum of nonradiative solutions decoupled by scaling, and thus the classification of nonradiative solutions is crucial in the proof of the soliton resolution conjecture for equation \eqref{NLW}.

We note that an analog of the rigidity theorem is known for the linear wave equation:
\begin{equation}
 \label{LW}
 \partial_t^2u-\Delta u=0.
\end{equation} 
Namely, by \cite{DuKeMe12}, if $u$ is a solution of \eqref{LW} and $N$  is odd,
\begin{equation}
 \label{equirepartition}
\sum_{\pm}\lim_{t\to \pm\infty}\int_{|x|>|t|}|\nabla_{t,x}u(t,x)|^2\,dx=\int |\nabla_{t,x}u(0,x)|^2\,dx.
 \end{equation} 
In particular, if $u$ satisfies \eqref{nonradi}, then $u=0$.

According to \cite{CoKeSc14}, \eqref{equirepartition} is not valid for solutions of the linear wave equation in even space dimension, even in the radial case. In this paper we will show however, as a consequence of \cite{CoKeSc14}, that the rigidity theorem for linear solutions is also valid when $N$ is even.
\begin{mainprop}
 \label{P:Neven}
Let $u$ be a solution of \eqref{LW} with $N\geq 3$. Assume \eqref{nonradi}. Then $u$ is the constant zero solution.
 \end{mainprop}
The proof of the rigidity theorem in space dimension $3$ for the nonlinear wave equation relies on an stronger exterior energy bound than \eqref{equirepartition}, proved in \cite{DuKeMe11a}, and which holds only for radial solutions. Assume $N=3$. Let $R>0$ and $u$ be a finite energy solution of \eqref{LW}. Then
\begin{equation}
 \label{linear_exterior}
 \sum_{\pm} \lim_{t\to\pm\infty} \int_{|x|>R+|t|} |\nabla_{t,x}u(t,x)|^2\,dx=\left\| \pi_{P(R)^{\bot}}\vec{u}(0)\right\|_{\HHH(R)}^2,
\end{equation} 
where $\HHH(R)$ is the space of restrictions of elements of $\HHH_{\rad}$ to $r\in [R,\infty)$, $P(R)=\vect\left\{(\frac{1}{r},0)\right\}$, and $\pi_{P(R)^{\bot}}$ is the orthogonal projection, in $\HHH(R)$, of the orthogonal of $P(R)$ in $\HHH(R)$. In the proof of the rigidity theorem in \cite{DuKeMe13}, it is used in a crucial way that $P(R)$ is of dimension $1$, and that the decay, as $r\to\infty$ of $1/r$ and $W(r)$ are of the same order at infinity.

An analog of the exterior energy bound mentioned above was proved in all odd space dimensions in \cite{KeLaLiSc15}. Namely, assume that $N\geq 3$ is odd, and let $P(R)$ be the subspace of $\HHH(R)$ spanned by
$$\PPP=\left\{\left(\frac{1}{r^{N-2k_1}},0\right),\left(0, \frac{1}{r^{N-2k_2}}\right),\quad 1\leq k_1\leq \ent{\frac{N+2}{4}},\; 1\leq k_2\leq \ent{\frac{N}{4}}\right\},$$
where $\ent{a}$ denotes the integer part of $a$. Then \eqref{linear_exterior} holds for all radial solutions $u_L$ of the linear wave equation \eqref{LW}. Note that 
$$P(R)=\left\{(u_0,u_1)\in \HHH(R),\quad \exists k\geq 1, \quad \Delta^k u_0=\Delta^ku_1=0\right\}.$$
The solutions of the linear wave equation \eqref{LW}, with initial data in $P(R)$ that are well-defined for $|x|>R+|t|$, can be computed explicitely, and one can check easily that for these solutions:
\begin{equation}
\label{nonradiR}
\lim_{t\to+\infty}\int_{|x|>R+|t|}|\nabla_{t,x}u(t,x)|^2\,dx=0. 
\end{equation} 
The space $P(R)$ is of dimension $\frac{N-1}{2}>1$ if $N\geq 5$, and for this reason, the proof of the rigidity theorem breaks down in these higher space dimensions. Note however that the exterior energy bound \eqref{linear_exterior} was used by C. Rodriguez \cite{Rodriguez16} to prove that the radial soliton resolution holds for a sequence of times in odd space dimensions.

In this paper, we show that the radial nonradiative solutions of \eqref{NLW} in odd space dimensions have very specific properties. Indeed, the initial data $(u_0,u_1)(r)$ of any radial solution $u$ of \eqref{NLW} that satisfies \eqref{nonradiR} for some $R$, has the same decay, for large $r$, as an element of $P(R)$. We also prove that the only radial nonradiative solutions of \eqref{NLW} in odd space dimension such that $(u_0,u_1)$ behaves like $(1/r^{N-2},0)$ at infinity are the stationary solutions $\pm \lambda^{\frac{N}{2}-1}W(\lambda x)$. Unfortunately, we were not able to prove the complete rigidity theorem, i.e. to exclude nonradiative solutions that behave at infinity as other elements of $P(R)$ than $(1/r^{N-2},0)$. However the results of this paper will be used, in a subsequent paper \cite{DuKeMe19Pb}, to prove the soliton resolution conjecture in all odd space dimensions. To state our main results, denote $\PPP=(\Xi_{k})_{1\leq k\leq m}$, where $m=\frac{N-1}{2}$ and $\Xi_k$ is one of the generators of $P(R)$, so that
$$ \left\|\Xi_k\right\|_{\HHH(R)}=\frac{c_k}{R^{k-\frac{1}{2}}}.$$
In particular, $\Xi_{m}=(\frac{1}{r^{N-2}},0)$. 

Adapting the well-posedness theory from \cite{GiVe95}, \cite{KeMe08} and \cite{BuCzLiPaZh13}, and using finite speed of propagation, we will work on solutions that are defined in the exterior of wave cones $\{|x|>R+|t|\}$. We refer to Section \ref{S:preliminaries} for the precise definition of these solutions.
\begin{theo} 
\label{T:maintheo}
Assume that $N\geq 5$ is odd.
Let $R_0>0$ and $u$ be a solution of \eqref{NLW} defined on $\{|x|>R_0+|t|\}$ such that 
\begin{equation}
\label{weaknonradi}
\sum_{\pm}\lim_{t\to\pm\infty}\int_{|x|>R_0+|t|}|\nabla_{t,x}u(t,x)|^2\,dx=0.
\end{equation} 
Then there exists $R_1\gg 1$, $k_0\in\llbracket 1,m\rrbracket$ and $\ell\in \RR$, with $\ell\neq 0$ if $k_0<m$ and such that, for all $t$,
\begin{equation*}
 \label{asymptotic_NR_intro}
 \forall R>R_1+|t|,\quad
\left\|\vec{u}(t)-\ell \,\Xi_{k_0}\right\|_{\HHH(R)}\lesssim \max\left(\frac{1}{R^{(k_0-\frac{1}2)\frac{N+2}{N-2}}},\frac{1}{R^{k_0+\frac 12}}\right),
\end{equation*} 
where the implicit constant depends on $u$.
\end{theo}
\begin{theo}
\label{T:maintheo2}
Let $u$ be as in the preceding theorem, and assume furthermore $k_0=m$. Then $u$ coincides with a stationary solution for large $r$. 
Furthermore, if $u$ is defined on $\{|x|>|t|\}$ and \eqref{weaknonradi} holds with $R_0=0$, then $u$ is stationary.
\end{theo}
The case of even space dimensions seems more difficult. Indeed, as proved in \cite{CoKeSc14}, the lower bound for the linear exterior energy fails in this case, but, as mentioned before, the qualitative result is still valid (see Proposition \ref{P:Neven}). However in order to pass to the nonlinear case in even dimension, we would need some quantitative version of the nonradiative property. 
We note however that it is possible to classify all nonradiative solutions with compactly supported initial data:
\begin{mainprop}
\label{P:nonradi_compact}
Let $N\geq 3$.
Let $u$ be a radial solution of \eqref{NLW} on $\{|x|>|t|\}$ such that $(u_0,u_1)$ is compactly supported. Assume \eqref{nonradi}. Then $u=0$.
\end{mainprop}
The proof of Proposition \ref{P:nonradi_compact} relies on a property of the linear wave equation with a potential, proved in our preceding work \cite{DuKeMe19Pa}. If also yields an interesting property of propagation of the support. If $u(t)$ is a solution of \eqref{NLW}, we denote by $\rho(t)$ the radius of the largest ball containing the essential support of $\vec{u}(t)$:
$$\rho(t)=\min\left\{ \rho\;:\; \int_{\rho}^{+\infty} (\partial_{t,r}u)^2\,r^{N-1}dr=0\right\}.$$
\begin{theo}
\label{T:support}
 Let $u$ be a radial solution of the linear wave equation \eqref{LW} or the nonlinear wave equation \eqref{NLW} such that $(u_0,u_1)$ is compactly supported. Then the following holds for all $t\geq 0$ or for all $t\leq 0$ in the domain of definition of $u$:
 \begin{equation}
 \label{propagation_support}
\rho(t)=\rho(0)+|t|.
 \end{equation} 
\end{theo}

The assumption that $(u_0,u_1)$ is radial is of course crucial in the preceding statements (except for Proposition \ref{P:Neven}). However some results are still available without a symmetry assumption: we refer to \cite{DuKeMe19Pa} for a rigidity theorem analogous to the one of \cite{DuKeMe13}, in space dimensions $3$ and $5$, close to the ground state $W$, and to \cite{DuJiKeMe17b} for the soliton resolution along a sequence of times in space dimensions $N\in \{3,4,5\}$.

The outline of the paper is as follows. In Section \ref{S:preliminaries}, we adapt the well-posedness theory to the exterior of wave cones and construct a profile decomposition adapted to this setting. This generalized well-posedness theory is needed to give a meaning to the statements of Theorems \ref{T:maintheo} and \ref{T:maintheo2}. We will also use it in our subsequent work \cite{DuKeMe19Pb} on soliton resolution in odd space dimension. It is quite delicate in space dimensions $N\geq 6$, where the standard Cauchy theory, developed in \cite{BuCzLiPaZh13}, uses spaces defined by fractional (and thus nonlocal) derivatives. 
In Section \ref{S:exterior}, we prove Proposition \ref{P:Neven}, Theorems \ref{T:maintheo} and \ref{T:maintheo2}.  The proof of Proposition \ref{P:Neven} is based on an asymptotic formula for even dimensions proved in \cite{CoKeSc14}, and an expansion in spherical harmonics to reduce to the radial case. The strategy of the proofs of the Theorems is quite similar to the one of the rigidity result in \cite{DuKeMe13}, using the general exterior energy bound for the linear wave equation proved in \cite{KeLaLiSc15}. However the analysis is much more complicated due to the higher dimension of the space $P(R)$ of singular solutions of the linear wave equation.
Proposition \ref{P:nonradi_compact} and Theorem \ref{T:support} are proved in Section \ref{S:compact}, using a property of the radial linear wave equation with a potential proved in \cite{DuKeMe19Pa}.

\section{Well-posedness and related issues}
\label{S:preliminaries}
\subsection{Notations}
\label{SS:notations}
If $\lambda>0$, $f\in \dot{H}^1$, we let $f_{(\lambda)}(x)=\frac{1}{\lambda^{\frac{N}{2}-1}}f\left( \frac{x}{\lambda} \right).$

If $A$ is a space of distributions on $\RR^N$, we will denote by $A_{\rad}$ the subspace of $A$ consisting of the elements of $A$ that are radial. We will, without making a distinction, consider a radial function as depending on the variable $x\in \RR^N$ or the variable $r=|x|$.

If $\Omega$ is an open subset of $\RR^{n}$, ($n=N$ or $n=N+1$), and $A=A(\RR^n)$ a Banach space of distributions on $\RR^n$, we recall that $A(\Omega)$ is the set of restrictions of elements of $A$ to $\Omega$, with the norm
$$\|u\|_{A(\Omega)}:=\inf_{\tilde{u}}\|\tilde{u}\|_{A(\RR^n)}.$$
where the infimum is taken over all $\tilde{u}\in A(\RR^n)$ such that $\tilde{u}_{\restriction \Omega}=u$. 
To lighten notation, if $R>0$ and $n=N$, we will denote by 
$$A(R):= A_{\rad}\left(\left\{x\in \RR^N\text{ s.t. }|x|>R\right\}\right).$$
We will mainly use this notation with $\HHH$, so that $\HHH(R)$ is the space of radial distributions $(u_0,u_1)$ defined for $r>R$ such that
$$ u_0\in L^{\frac{2N}{N-2}}((R,+\infty), r^{N-1}dr), \quad \int_R^{\infty} (\partial_ru_0)^2r^{N-1}\,dr<\infty$$
and 
$$ u_1\in L^2\left((R,+\infty),r^{N-1}dr\right).$$

We will often consider solutions of the wave equation in the exterior of wave cones. For $R>0$, we denote 
$$\Gamma_R(t_0,t_1)=\{|x|> R+|t|,\; t\in [t_0,t_1]\}.$$
To lighten notations, we will denote
$$\Gamma_R(T)=\Gamma_R(0,T),\quad \Gamma_R=\Gamma_{R}(0,\infty).$$
We denote by $S_L(t)$ the linear wave group: 
\begin{equation}
\label{SL}
S_L(t)(u_0,u_1)=\cos(t\sqrt{-\Delta})u_0+\frac{\sin(t\sqrt{-\Delta})}{\sqrt{-\Delta}}u_1,
\end{equation} 
so that the general solution (in the Duhamel sense) of 
\begin{equation}
 \label{LW_in}
\left\{
\begin{aligned}
(\partial_t^2-\Delta)u&=f\\
\vec{u}_{\restriction t=t_0}&=(u_0,u_1)\in \HHH,
\end{aligned}\right.
\end{equation} 
where $I$ is an interval and $t_0\in I$
is 
\begin{equation}
\label{Duhamel}
u(t)=S_L(t-t_0)(u_0,u_1)+\int_{t_0}^{t}S_L(t-s)(0,f(s))\,ds. 
\end{equation} 
We note that by finite speed of propagation, the restriction of $u$ to $\Gamma_R(T)$ depends only on the restriction of $f$ to $\Gamma_R(T)$ and the restriction of $(u_0,u_1)$ to $\{r>R\}$.

\subsection{Function spaces and Strichartz estimates}
We will denote by $\dot{W}^{s,p}(\RR^N)$ the homogeneous Sobolev space defined as the closure of $C^{\infty}_0(\RR^N)$ with respect to the norm $\|\cdot\|_{\dot{W}^{s,p}}$ defined by
$$ \|f\|_{\dot{W}^{s,p}}:=\|D^s f\|_{L^p},$$
where $D^s$ is the Fourier multiplier of symbol $|\xi|^{s}$. We denote by $\dot{B}^{s}_{p,q}$ the standard homogeneous Besov space, which can be defined using Littlewood-Paley decomposition or the real interpolation method: $\dot{B}^{s}_{p,q}=\left[L^p,\dot{W}^{1,p}\right]_{s,q}$, $0<s<1$, $1\leq p,q\leq \infty.$

We define, following \cite{BuCzLiPaZh13}:
\begin{gather*}
S:= L^{\frac{2(N+1)}{N-2}}(\RR^{1+N}),\\ 
W:=L^{\frac{2(N+1)}{N-1}}\left( \RR,\dot{B}^{\frac{1}{2}}_{\frac{2(N+1)}{N-1},2} (\RR^N) \right)\\
W':= L^{\frac{2(N+1)}{N+3}}\left(\RR,\dot{B}^{\frac{1}{2}}_{\frac{2(N+1)}{N+3},2}(\RR^N) \right)\\
X:=L^{\frac{N^2+N}{N+2}}\left(\RR,\dot{W}^{\frac{2}{N},\frac{2(N+1)}{N-1}}(\RR^N)\right)\\
X':=L^{\frac{N^2+N}{3N+2}}\left(\RR,\dot{W}^{\frac{2}{N},\frac{2(N+1)}{N+3}}(\RR^N)\right).
\end{gather*}

If $I$ is an interval, we will denote by $S(I)$, $W(I)$, $W'(I)$, $X(I)$, $X'(I)$ the restriction of these spaces to $I\times \RR^N$.

We will need the following Strichartz estimates (see \cite{St77a} \cite{GiVe95}): if $t_0\in I$, $f\in W'(I)$,  $(u_0,u_1)\in \HHH$, then $u$ (defined by \eqref{Duhamel}) is in $S(I)\cap W(I)$ and
\begin{equation}
\label{Strichartz}
\sup_{t\in \RR} \|\vec{u}(t)\|_{\HHH}+\|u\|_{S(I)}+
\|u\|_{W(I)} 
\lesssim \|(u_0,u_1)\|_{\HHH(I)} +\|f\|_{W'(I)}.
\end{equation}
We denote $F(u)=|u|^{\frac{4}{N-2}}u$. We will need the following chain rule for fractional derivatives (see \cite[Lemma 2.10]{BuCzLiPaZh13}): for a function $u\in \dot{B}^{\frac{1}{2}}_{\frac{2(N+1)}{N-1},2}(\RR^N)\cap L^{\frac{2(N+1)}{N-2}}(\RR^N)$,
\begin{equation}
\label{fxt_frct}
\|F(u)\|_{\dot{B}^{\frac{1}{2}}_{\frac{2(N+1)}{N+3},2}} \lesssim \|u\|_{\dot{B}^{\frac{1}{2}}_{\frac{2(N+1)}{N-1},2}}\|u\|_{L^{\frac{2(N+1)}{N-2}}}^{\frac{4}{N-2}}.
\end{equation}
combining with H\"older estimates, we obtain, considering a function $u\in W(I)\cap S(I)$
\begin{equation}
 \label{fractional}
\left\|F(u)\right\|_{W'(I)}\lesssim \|u\|_{S(I)}^{\frac{4}{N-2}}\|u\|_{W(I)}.
\end{equation}

\subsection{Local and global Cauchy theory}
\label{SS:defi}
\begin{defi}
\label{D:solution}
 Let $I$ be an interval with $t_0\in I$, $(u_0,u_1)\in \HHH$. If $N\geq 6$, we call solution of \eqref{NLW} on $I\times \RR^N$, with initial data
\begin{equation}
 \label{S0}
\vec{u}_{\restriction t=t_0}=(u_0,u_1)
\end{equation} 
a function $u\in C^0(I,\hdot)$ such that $\partial_t u\in C^0(I,L^2)$ and
\begin{equation}
 \label{S1}
\forall t\in I,\quad u(t)=S_L(t-t_0)(u_0,u_1)+\int_{t_0}^tS_L(s-t_0)F(u(s))\,ds.
\end{equation} 
If $N\in \{3,4,5\}$, a solution is defined in the same way, with the additional requirement that $u\in S(J\times \RR^N)$ for all compact intervals $J\subset I$.
\end{defi}
\begin{remark}
 The uniqueness of Duhamel solutions of \eqref{NLW} in the class $C^0(I,\HHH)$ was proved in \cite{BuCzLiPaZh13} in dimension $N\geq 6$. This unconditional uniqueness is not known in space dimensions $3$, $4$ and $5$, which explains the additional requirement $u\in S(J\times \RR^N)$ in the definition.
\end{remark}

It is known (see \cite{GiVe95}, \cite{KeMe08} and \cite{BuCzLiPaZh13}), that for all initial data $(u_0,u_1)$, there is a unique maximal solution $u$ defined on a maximal interval $(T_-,T_+)$ and that it satisfies the following blow-up criterion:
$$T_+<\infty\Longrightarrow \|u\|_{S([t_0,T_+))}=\infty.$$


We will also need to define a solution on the exterior $\Gamma_R(t_0,t_1)$ of wave cones. We will use the following continuity property of multiplication by characteristic functions on Besov space:
\begin{lemma}
 \label{L:char_Besov}
 Let $R\geq 0$.
 \begin{itemize}
  \item The multiplication by the characteristic function $\indic_{\{|x|>R\}}$ is 
a continuous function from $\dot{B}^{\frac 12}_{\frac{2(N+1)}{N+3},2}(\RR^N)$ into itself, and from
$\dot{W}^{\frac{2}{N},\frac{2(N+1)}{N+3}}(\RR^N)$ into itself. In both cases, the operator norm is independent of $R$.
\item Let $I$ be an interval. The multiplication by the characteristic function $\indic_{\{|x|>R+|t|\}}$ is continuous from $W'(I)$ into itself and from $X'(I)$ into itself. The operator norm is independent of $R$ and $I$.
\end{itemize}
 \end{lemma}
\begin{proof}
 The second point follows immediately from the first point.
 
 By elementary scaling argument, it is sufficient to consider only the case $R=1$. Since $\indic_{\{|x|<1\}}=1-\indic_{\{|x|>1\}}$ almost everywhere, we are reduced to prove the boundedness of the multiplication by the characteristic function of the unit ball. According to \cite[Proposition 3.3.2]{Triebel83}, if 
 $\frac{1}{p}-1<s<\frac{1}{p}$, then the multiplication by $\indic_{\{|x|<1\}}$ is continuous from $B^{s}_{p,2}$ to $B^{s}_{p,2}$ and from $W^{s,p}$ to $W^{s,p}$ (to obtain this second fact from Proposition 3.3.2 of \cite{Triebel83}, recall that $W^{s,p}$ is identical to the Triebel space $F^{s}_{p,2}$). This proves that the multiplication by $\indic_{\{|x|<1\}}$ is a bounded operator on the non-homogeneous analogs of the homogeneous spaces that we want to consider here. 
 
 Let $\chi\in C_0^{\infty}(\RR^N)$ be such that $\chi(x)=1$ if $|x|\leq 1$. Then $\indic_{\{|x|<1\}}=\indic_{\{|x|<1\}}\chi$. Furthermore, by an explicit computation and the Sobolev inequality, we have that for all $p\geq 1$, the multiplication by $\chi$ is a bounded operator from $L^p$ to $L^p$ and from $\dot{W}^{1,p}$ to $W^{1,p}$. By real or complex interpolation (Theorems 6.2.4, 6.3.1, 6.4.5 in \cite{BerghLofstrom76BO}) we deduce that for all $s\in (0,1)$ the multiplication by $\chi$ is a bounded operator from $\dot{B}^{s}_{p,2}$ to $B^{s}_{p,2}$ and from $\dot{W}^{s,p}$ to $W^{s,p}$. This concludes the proof.
\end{proof}
\begin{remark}
\label{R:frxt}
Combining Lemma \ref{L:char_Besov} with the fractional chain rule \eqref{fxt_frct}, we see that 
\begin{equation*}
\left\|\indic_{\Gamma_R(T)} F(u)\right\|_{\dot{B}^{\frac{1}{2}}_{\frac{2(N+1)}{N+3},2}} 
\lesssim \|u\|_{\dot{B}^{\frac{1}{2}}_{\frac{2(N+1)}{N-1},2}}\|u\|_{L^{\frac{2(N+1)}{N-2}}}^{\frac{4}{N-2}}.
\end{equation*}
We claim that we can replace $\|u\|_{L^{\frac{2(N+1)}{N-2}}}^{\frac{4}{N-2}}$ in the right hand side by $\|\indic_{\{|x|>R\}}u\|_{L^{\frac{2(N+1)}{N-2}}}^{\frac{4}{N-2}}$, which yields, after an integration in time, the following analog of \eqref{fractional} for outside wave cones:
\begin{equation}
 \label{fractional_cones}
 \|\indic_{\Gamma_R(T)}F(u)\|_{W'((0,T))}\lesssim \|u\|_{S(\Gamma_R(T))}^{\frac{4}{N-2}}\|u\|_{W((0,T))}.
\end{equation} 
By scaling, we can assume $R=1$.
Recall from Subsection \ref{SS:notations} the definition of $\dot{B}^{\frac{1}{2}}_{\frac{2(N+1)}{N-1},2}(\{|x|>1\})$.
There exists an extension operator $\EEE$ which is bounded from 
$\dot{B}^{\frac{1}{2}}_{\frac{2(N+1)}{N-1},2}(\{|x|>1\})$ to  $\dot{B}^{\frac{1}{2}}_{\frac{2(N+1)}{N-1},2}(\RR^N)$ and from $L^{\frac{2(N+1)}{N-2}}(\{|x|>1\})$ to $L^{\frac{2(N+1)}{N-2}}(\RR^N)$ (see e.g. \cite[Theorem 3.3.4]{Triebel83}, which is stated on a bounded domain but remains valid with the same proof on the exterior of a smooth bounded compact set). We have
\begin{multline*}
\|\indic_{\{|x|>1\}}F(u)\|_{\dot{B}^{\frac{1}{2}}_{\frac{2(N+1)}{N+3},2}} =
\|\indic_{\{|x|>1\}}F(\EEE u)\|_{\dot{B}^{\frac{1}{2}}_{\frac{2(N+1)}{N+3},2}}\\
\lesssim \|\EEE u\|_{\dot{B}^{\frac{1}{2}}_{\frac{2(N+1)}{N-1},2}}\|\EEE u\|_{L^{\frac{2(N+1)}{N-2}}}^{\frac{4}{N-2}}\lesssim \|u\|_{\dot{B}^{\frac{1}{2}}_{\frac{2(N+1)}{N-1},2}}\|\indic_{\{|x|>1\}}u\|_{L^{\frac{2(N+1)}{N-2}}}^{\frac{4}{N-2}}
\end{multline*}
\end{remark}

\begin{defi}
\label{D:sol_cone}
 Let $t_0<t_1$, $R\geq 0$. Let $(u_0,u_1)\in \HHH(R)$. A solution $u$ of \eqref{NLW} on $\Gamma_R(t_0,t_1)$ with initial data $(u_0,u_1)$ is the restriction to $\Gamma_R(t_0,t_1)$ of a solution $\tilde{u}\in C^0([t_0,t_1],\hdot)$ with $\partial_t\tilde{u}\in C^0([t_0,t_1],L^2)$, to the equation:
 \begin{equation}
  \label{NLW_trunc}
\partial_t^2\tu-\Delta \tu=|\tu|^{\frac{4}{N-2}}\tu\indic_{\{|x|>R+|t|\}},
\end{equation} 
with an initial data
\begin{equation}
 \label{ID_trunc}
 \vec{\tu}_{\restriction t=t_0}=(\tilde{u}_0,\tilde{u}_1),
\end{equation}
where $(\tilde{u}_0,\tilde{u}_1)\in \HHH$ is an extension of $(u_0,u_1)$
\end{defi}
Note that by finite speed of propagation, the value of $u$ on $\Gamma_R(t_0,t_1)$ does not depend on the choice of $(\tilde{u}_0,\tilde{u}_1)$, provided $(\tilde{u}_0,\tilde{u}_1)$ and $(u_0,u_1)$ coincide for $r>R$.

According to Lemma \ref{L:char_Besov}, the Cauchy theory in \cite{BuCzLiPaZh13} (or \cite{KeMe08} for the case $N\in\{3,4,5\}$) adapts easily to the case of solutions outside wave cones. We give the statements, and omit most proofs that are the same as in \cite{KeMe08}, \cite{BuCzLiPaZh13}. The space $S(\Gamma_R(T))$ in the following proposition is defined in Subsection \ref{SS:notations}.

\begin{prop}[Local well-posedness]
 \label{P:LWP_cone}
 Let $R\geq 0$, $(u_0,u_1)\in \HHH(R)$ and $T>0$. Assume   
 $$\|(u_0,u_1)\|_{\HHH(R)}\leq A.$$
 Then there exists $\eta=\eta(A)$ such that if 
 $$\|S_L(t)(u_0,u_1)\|_{S(\Gamma_R(T))}<\eta,$$
 then there exists a unique solution $u$ to \eqref{NLW} on $\Gamma_R(T)$. Furthermore for all $t\in [0,T]$,
$$\|\vec{u}(t)-\vec{S}_L(t)(u_0,u_1)\|_{\HHH(R+|t|)}\leq C\eta^{\theta_N}A^{1-\theta_N}$$
for some constant $\theta_N$, $0<\theta_N<1$, depending only on $N$.
\end{prop}
(See \cite[Theorem 3.3]{BuCzLiPaZh13}).
\begin{remark}
\label{R:LWP_cone}
On $\RR^N$, the unconditional uniqueness of $C^0(I,\HHH)$ solutions is only known in dimension $N\geq 6$ (see \cite{BuCzLiPaZh13}). In the radial setting, outside of a wave cone, it can be proved easily in any space dimension, using the radial Sobolev inequality. Indeed, for a radial continuous function $u$ defined on $\Gamma_R(T)$ with $R>0$, we define
 $$ \|u\|_{R,T}=\sup_{(t,r)\in \Gamma_R(T)} r^{\frac{N}{2}-1} |u(t,r)|<\infty,$$
 and let
 $$X(R,T)=\left\{u\in C_{\rad}^0(\Gamma_R(T)),\; \|u\|_{R,T}<\infty\right\}.$$
 Using the radial Sobolev inequality, we see that if $(u_0,u_1)\in \HHH(R)$, $(\tilde{u}_0,\tilde{u}_1)\in \HHH$ is an extension of $(u_0,u_1)$ and $u_L(t)=S_L(t)(\tu_0,\tu_1)$, then $u_L\in X(R,T)$, and $\|u_L\|_{R,T}\lesssim \|(u_0,u_1)\|_{\HHH(R)}=:M$. For $v\in X(R,T)$, let
 $$\Phi(v)(t)=u_L(t)+\int_0^t \frac{\sin((t-s)\sqrt{-\Delta})}{\sqrt{-\Delta}}\left(\indic_{\Gamma_R(T)}|v(s)|^{\frac{4}{N-2}}v(s)\right)\,ds.$$
 Using the radial Sobolev inequality, energy inequalities, and finite speed of propagation, we obtain:
 \begin{multline*}
\left\|\int_0^t \frac{\sin((t-s)\sqrt{-\Delta})}{\sqrt{-\Delta}}\left(\indic_{\Gamma_R(T)}|v(s)|^{\frac{4}{N-2}}v(s)\right)\,ds\right\|_{X(R,T)}\lesssim \left\| \indic_{\Gamma_R(T)} |v|^{\frac{N+2}{N-2}}\right\|_{L^1(0,T,L^2)}\\
\lesssim \|v\|^{\frac{N+2}{N-2}}_{R,T}\log \left( 1+\frac{T}{R} \right).
\end{multline*}
As a consequence, we see that if $\frac{T}{R}\leq \exp \left( \frac{1}{CM^{\frac{4}{N-2}}} \right)-1$, the ball 
$$\left\{v\in X(R,T),\;\|v\|_{R,T}\leq 2M\right\}$$
is stable by $\Phi$. A similar argument proves that $\Phi$ is a contraction on this ball, which proves the uniqueness statement.
 \end{remark}
Gluing the preceding local solutions, we obtain a maximal solution defined on a maximal domain $\Gamma_R(0,T_R^+)$. 
By the Remark \ref{R:LWP_cone}, 
$$T_R^+\geq R \left(\exp\left( \frac{1}{C M^{\frac{4}{N-2}}} \right)-1\right),$$
where $M=\|(u_0,u_1)\|_{\HHH(R)}$. Iterating this remark, we obtain the following blow-up criterion
\begin{equation}
 \label{b_up_1}
T_R^+<\infty \Longrightarrow \lim_{t\to T_R^+} \|\vec{u}(t)\|_{\HHH(R+t)}=+\infty.
\end{equation} 
We can also write a blow-up criterion in term of space-time norms:
\begin{lemma}
\label{L:scattering}
Assume $u\in S\left(\Gamma_{R}(T_R^+)\right)$. Then $u$ is global. Furthermore, $u$ scatters to a linear solution for $\{|x|>R+|t|\}$: there exists a solution $v_L$ of the linear wave equation on $\RR\times \RR^N$ such that 
$$\lim_{t\to +\infty} \left\|\vec{u}(t)-\vec{v}_L(t)\right\|_{\HHH(R+|t|)}=0.$$
\end{lemma}
\begin{proof}
\setcounter{step}{0}
\begin{step}
Slightly abusing notations, we denote by $u$ the solution to the equation \eqref{NLW_trunc} with an initial data $(\tilde{u}_0,\tilde{u}_1)\in \HHH_{\rad}$ at $t=0$ such that 
$$ (\tu_0,\tu_1)(r)=(u_0,u_1)(r),\quad r>R.$$

 We first prove that $u\in W\left( [0,T_R^+]\right)$. For any $t\in [0,T_R^+)$, we let
 $$\varphi(t)=\|u(t)\|_{\dot{B}^{\frac 12}_{\frac{2(N+1)}{N-1},2}},\quad \psi(t)=\|u(t)\|^{\frac{4}{N-2}}_{L^{\frac{2(N+1)}{N-2}}\left(\{|x|>R+|t|\}\right)},$$
 and we note that 
 $$\forall T\in (0,T_+), \quad \|\varphi\|_{L^{\frac{2(N+1)}{N-1}}(0,T)}=\|u\|_{W\left(0,T\right)}$$ and that $\|\psi\|_{L^{\frac{N+1}{2}}(0,T^+_R)}=\|u\|_{S\left( \Gamma_R(0,T_R^+) \right)}^{\frac{4}{N-2}}$ (which is finite by our assumptions). 
 By Strichartz estimates, 
 $$\forall T\in (0,T_R^+),\quad \|u\|_{W(0,T)}\lesssim \|(u_0,u_1)\|_{\HHH(R)}+\left\|\indic_{\Gamma_R(0,T)}F(u)\right\|_{W'}.$$
 Using Lemma \ref{L:char_Besov} and the fractional chain rule \eqref{fxt_frct} (together with Remark \ref{R:frxt}), we deduce
$$ \forall T\in [0,T_R^+), \quad \|\varphi\|_{L^{\frac{2(N+1)}{N-1}}(0,T)}\lesssim \|(u_0,u_1)\|_{\HHH(R)}+\|\psi \varphi\|_{L^{\frac{2(N+1)}{N+3}}(0,T)}.$$
Using a variant of Gr\"onwall's inequality (see e.g. \cite[Lemma 8.1]{FaXiCa11}), we deduce 
$$\varphi\in L^{\frac{2(N+1)}{N-1}}(0,T_+), \quad \|\varphi\|_{L^{\frac{2(N+1)}{N-1}} (0,T_+)}\leq C_M\|(u_0,u_1)\|_{\HHH(R)},$$
for some constant $C_M$ depending only on $M=\|u\|_{S(\Gamma_R(T_R^+))}$.
\end{step}
\begin{step}[Global existence]
 By the preceding step and the fractional chain rule
\eqref{fxt_frct}, we obtain $F(u)\in W'(0,T_+)$. Combining with Strichartz inequalities, we deduce
$$\limsup_{t\to T_R^+}\|\vec{u}(t)\|_{\HHH(R+|t|)}<\infty,$$
which, by the blow-up criterion \eqref{b_up_1}, is sufficient to ensure that $T_R^+=+\infty$.
 \end{step}
\begin{step}[Scattering]
 According to the preceding steps,
 $$\partial_t^2u-\Delta u=F(u)\indic_{\{|x|>R+|t|\}}\in W'(0,\infty).$$
Let $U(t)=S_L(-t)\vec{u}(t)$.
Using the dual of the Strichartz estimates \eqref{Strichartz}, we see that the preceding equation implies that $\vec{U}(t)$ has a limit $(v_0,v_1)$ in $\HHH$ as $t\to+\infty$. Letting $v_L(t)=S_L(t)(v_0,v_1)$, we obtain
$$ \lim_{t\to+\infty}\left\|\vec{\tu}(t)-\vec{v}_L(t)
\right\|_{\HHH}=0,$$
which yields the desired conclusion.
\end{step}

\end{proof}

We will also need the following long-time perturbation theory result (see \cite[Theorem 2.20]{KeMe08}, \cite[Theorem 3.6]{BuCzLiPaZh13}, \cite[Proposition A.1]{Rodriguez16}).
\begin{prop}
\label{P:LTPT}
Let $A>0$. There exists $\eta_0=\eta_0(A)$ with the following property.
 Let $R>0$, $T\in (0,\infty]$, $(u_0,u_1)\in \HHH(R)$ and $(v_0,v_1)\in \HHH(R)$. Assume that $v$ is a restriction to $\Gamma_R(0,T)$ of a function $V$ such that $\vec{V}\in C^0([0,T],\HHH)$ and 
 $$\partial_t^2V-\Delta V=\indic_{\{|x|>R+|t|\}}\left(F(V)+e_1+e_2\right),$$
 with 
 \begin{gather*}
\sup_{0\leq t\leq T}\|V(t)\|_{\HHH(R+|t|)}+\|V\|_{W(0,T)}\leq A\\
\|(u_0,u_1)-(v_0,v_1)\|_{\HHH(R)}+\|e_1\|_{W'(0,T)}+\|e_2\|_{L^1((0,T),L^2)}=\eta\leq \eta_0, 
\end{gather*}
Then the solution with initial data $(u_0,u_1)$ is defined on $\Gamma_R(T)$ and 
$$ \|v-u\|_{S(\Gamma_R(T))}\leq C\eta^{c_N},$$
for some constant $c_N\in (0,1]$ depending only on $N\geq 3$.
\end{prop}
\begin{remark}
 In \cite{KeMe08,BuCzLiPaZh13,Rodriguez16}, $e_2=0$, but the argument easily adapts to the setting of Proposition \ref{P:LTPT}.
\end{remark}

\subsection{Profile decomposition}
\label{SS:profile}
Let $\big\{(u_{0,n},u_{1,n})\big\}_n$ be a bounded sequence of radial functions in $\HHH$. We say that it admits a profile decomposition if for all $j\geq 1$, there exist a solution $U^j_F$ to the free wave equation with initial data in $\HHH$ and sequences of parameters $\{\lambda_{j,n}\}_n\in (0,\infty)^{\NN}$, $\{t_{j,n}\}_n\in \RR^{\NN}$ such that
\begin{equation}
 \label{psdo_orth}
 j\neq k \Longrightarrow \lim_{n\to\infty}\frac{\lambda_{j,n}}{\lambda_{k,n}}+\frac{\lambda_{k,n}}{\lambda_{j,n}}+\frac{|t_{j,n}-t_{k,n}|}{\lambda_{j,n}}=+\infty,
\end{equation} 
and, denoting 
\begin{gather}
 \label{rescaled_lin}
 U^j_{F,n}(t,r)=\frac{1}{\lambda_{j,n}^{\frac N2-1}}U^j_F\left( \frac{t-t_{j,n}}{\lambda_{j,n}},\frac{r}{\lambda_{j,n}} \right),\quad j\geq 1\\
 w_{n}^J(t)=S_L(t)(u_{0,n},u_{1,n})-\sum_{j=1}^J U^j_{F,n}(t),
\end{gather} 
one has 
\begin{equation}
 \label{wnJ_dispersive}
 \lim_{J\to\infty}\limsup_{n\to\infty}\|w_n^J\|_{S(\RR)}=0.
\end{equation} 
We recall (see \cite{BaGe99}, \cite{Bulut10}) that any bounded sequence in $\HHH$ has a subsequence that admits a profile decomposition. We recall also that the properties above imply that the following weak convergences hold:
\begin{equation}
 \label{wlim_w}
 j\leq J\Longrightarrow \left( \lambda_{j,n}^{\frac{N}{2}-1}w_n^J\left(t_{j,n},\lambda_{j,n}\cdot \right),\lambda_{j,n}^{\frac{N}{2}}\partial _tw_n^J\left(t_{j,n},\lambda_{j,n}\cdot \right)\right) \xrightharpoonup[n\to\infty]{} 0 \text{ in }\HHH.
\end{equation} 

If $\{(u_{0,n},u_{1,n})\}_n$ admits a profile decomposition, we can assume, extracting subsequences and time-translating the profiles if necessary that the following limit exists:
$$\lim_{n\to\infty}\frac{-t_{j,n}}{\lambda_{j,n}}=\tau_j\in \{-\infty,0,\infty\}.$$
Using the existence of wave operator for the equation \eqref{NLW} if $\tau_j\in \{\pm\infty\}$ or the local well-posedness if $\tau_j=0$, we define the nonlinear profile $U^j$ associated to $\left(U^j_F,\{\lambda_{j,n}\}_n,\{t_{j,n}\}_n\right)$ as the unique solution to the nonlinear wave equation \eqref{NLW} such that 
$$\lim_{t\to\tau_j} \left\|\vec{U}^j(t)-\vec{U}^j_F(t)\right\|_{\HHH}=0.$$
We also denote by $U^j_n$ the rescaled nonlinear profile:
$$ U^j_n(t,r)=\frac{1}{\lambda_{j,n}^{\frac N2-1}}U^j\left( \frac{t-t_{j,n}}{\lambda_{j,n}},\frac{r}{\lambda_{j,n}} \right).$$
Then we have the following superposition principle outside the wave cone $\Gamma_0:=\left\{(t,x)\in \RR\times \RR^N\; : \;|x|>t>0\right\}.$
\begin{prop}
 \label{P:NL_profile}
 Let $\{(u_{0,n},u_{1,n})\}_n$ be a bounded sequence of in $\HHH_{\rad}$. Assume that for all $j$ such that $\tau_j=0$, the nonlinear profile $U^j$ can be extended to a solution on $\Gamma_0$ (in the sense of Definition \ref{D:sol_cone}) such that $U^j\in S(\Gamma_0)$. Then for large $n$, there is a solution $u_n$ defined on $\Gamma_0$ with initial data $\{(u_{0,n},u_{1,n})\}_n$ at $t=0$. Furthermore, denoting, for $J\geq 1$, $(t,r)\in \Gamma_0$
 $$R_n^J(t,r)=u_n(t,r)-\sum_{j=1}^J U_n^j(t,r)-w_n^J(t,r),$$
 we have 
 $$\lim_{J\to\infty} \lim_{n\to\infty}\left[\|R_n^J\|_{S(\Gamma_0)}+\sup_{t\geq 0}\left\|\vec{R}_n^J(t)\right\|_{\HHH(t)}\right]=0.$$
\end{prop}
We omit the proof, which is similar to the proof when the solution is not restricted to the exterior of a wave cone (see \cite[Proposition 2.3]{Rodriguez16}).

Let us emphasize the fact that under the assumptions of Proposition \ref{P:NL_profile}, the profiles $U_n^j(t,r)$ are well-defined on $\Gamma_0$. 
\begin{itemize}
 \item if $\tau_j=0$, then $U^j_n(t,r)=\frac{1}{\lambda_{j,n}^{\frac{N}{2}-1}}U^j\left(\frac{t}{\lambda_{j,n}},\frac{r}{\lambda_{j,n}}\right)$, which is defined on $\Gamma_0$ since by assumption $U^j$ is defined on $\Gamma_0$. 
 \item if $\tau_j=+\infty$, then by definition $U^j$ is globally defined in the future, so that $U^j_n$ is well defined for $t\geq 0$ and large $n$.
 \item if $\tau=-\infty$, we know that $U^j$ is globally defined in the past. Let $T_+$ be the maximal interval existence of $U^j$. If $T_+=+\infty$, then $U^j_n$ is of course defined on $\Gamma_0$. If $T_+$ is finite, we can take $R$ large so that, $\|\vec{U}^j(T_+-1)\|_{\HHH(R)}$ is small, so that by the small data well-posedness theory (Proposition \ref{P:LWP_cone}), $U^j$ is defined on the set $\Gamma_R(T_+-1,\infty)$. For large $n$, we have
 \begin{multline*}
 \left\{ \left( \frac{t-t_{j,n}}{\lambda_{j,n}},\frac{r}{\lambda_{j,n}} \right),\; r>t>0\right\}=\left\{(s,\rho),\; \rho\geq \tau +\frac{t_{j,n}}{\lambda_{j,n}}\geq 0\right\}
 \\
 \subset \Gamma_R(T_+-1,\infty) \cup \Big((-\infty,T_+)_{\tau}\times (0,\infty)_{\rho}\Big),
 \end{multline*}
which shows again that $U^j_n(t,r)$ is well-defined for $(t,r)\in \Gamma_0$ and large $n$.
\end{itemize}
\section{Non-radiative solutions}
\label{S:exterior}

In this section we prove the main results of the paper: Proposition \ref{P:Neven} and Theorems \ref{T:maintheo} and \ref{T:maintheo2}.
After preliminaries on the free wave equations, we state and prove in subsection \ref{SS:asympt} three results that imply Theorems \ref{T:maintheo} and \ref{T:maintheo2}. 
\begin{defi}
\label{D:non-radiative}
Let $u$ be a solution of the nonlinear wave equation \eqref{NLW} (or another wave equation considered in this paper), with initial data at $t=0$. We say that $u$ is \emph{non-radiative} if $u$ is defined on $\{|x|>|t|\}$
$$\sum_{\pm} \lim_{t\to\pm \infty} \int_{|x|\geq |t|}|\nabla_{t,x}u(t,x)|^2\,dx=0.$$
We say that $u$ is \emph{weakly non-radiative} if for large $R>0$, $u$ is defined on $\{|x|>|t|+R\}$.
$$\sum_{\pm} \lim_{t\to\pm \infty} \int_{|x|\geq |t|+R}|\nabla_{t,x}u(t,x)|^2\,dx=0.$$
\end{defi}

\subsection{Nonradiative solutions for the free wave equation in odd space dimension}
When $N\geq 3$ is odd, non-radiative and weakly nonradiative solution of the free wave equation \eqref{LW}
are well-understood. As recalled in the introduction, it follows from the equirepartition of the energy\eqref{equirepartition} proved in \cite{DuKeMe12},
that the only non-radiative solution of \eqref{LW} is zero. 

The bound \eqref{linear_exterior} 
proved in \cite{KeLaLiSc15}, implies that the weakly non-radiative solutions are the ones that coincide with elements of $P(R)$ for large $R$, where $P(R)$ is defined as the subspace of $\HHH(R)$ spanned by
$$\PPP=\left\{\left(\frac{1}{r^{N-2k_1}},0\right),\left(0, \frac{1}{r^{N-2k_2}}\right),\quad 1\leq k_1\leq \ent{\frac{N+2}{4}},\; 1\leq k_2\leq \ent{\frac{N}{4}}\right\},$$
where $\ent{a}$ denotes the integer part of $a$. 

One can check that the dimension of $P(R)$ is exactly 
$$m=\frac{N-1}{2}.$$ Furthermore, by direct computations
\begin{align*}
\left\|\left( \frac{1}{r^{N-2k_1}},0 \right)\right\|^2_{\HHH(R)}&=\frac{(N-2k_1)^2}{(N-4k_1+2)R^{N-4k_1+2}},\\
\left\|\left(0, \frac{1}{r^{N-2k_2}} \right)\right\|^2_{\HHH(R)}&=\frac{1}{(N-4k_2)R^{N-4k_2}}
\end{align*} 
As in the introduction, we denote the elements of $\PPP$ as $(\Xi_{k})_{k\in \llbracket 1,m\rrbracket}$, choosing $\Xi_k$ so that
\begin{equation}
 \label{normXi}
 \left\|\Xi_k\right\|_{\HHH(R)}=\frac{c_k}{R^{k-\frac{1}{2}}},
\end{equation} 
for some constant $c_k\neq 0$. Thus $\Xi_m(r)=\left(r^{2-N},0\right)$, and
$$\Xi_1(r)=
\begin{cases}
\left(r^{-m},0\right)&\text{ if }m\text{ is odd},\\
\left(0,r^{-(m+1)}\right)&\text{ if }m\text{ is even.}
\end{cases}
$$
The norm of an element of $P(R)$ in $\HHH(R)$ is equivalent to the sum of the absolute values of its coordinates in $\PPP$. This is uniform with respect to $R$, up to some powers of $R$:
\begin{claim}
\label{Cl:coord}
 Let $U\in P(R)$ and denote by $\left( \theta_k(R) \right)_{1\leq k\leq m}$ its coordinates in $\PPP$. Then
 $$\left\|U\right\|_{\HHH(R)}\approx \sum_{k=1}^m\frac{\left|\theta_{k}(R)\right|} {R^{k-1/2}},$$
 where the implicit constant is independent of $R>0$.
\end{claim}
\begin{proof}
Using the equivalence of norms in finite dimension, we see that if $U\in P(1)$
\begin{equation}
\label{normP1}
\|U\|_{\HHH(1)}\approx \sum_{k=1}^m |\theta_k(1)|. 
\end{equation} 
We now assume that $U=(u,v)\in \HHH(R)$, and consider:
$$U_{R^{-1}}(x)=\left(R^{\frac{N}{2}-1}u(Rx),R^{\frac{N}{2}}v(Rx)\right).$$
Note that $U_{R^{-1}}\in \HHH(1)$, and that 
$$ \left\|U_{R^{-1}}\right\|_{\HHH(1)}=\left\| U\right\|_{\HHH(R)}.$$
We see in particular the 
$$ \frac{c_k}{R^{k-1/2}}=\left\|\Xi_k\right\|_{\HHH(R)}=\left\|(\Xi_{k})_{R^{-1}}\right\|_{\HHH(1)},$$
which implies, using the homogeneity of $\Xi_k$,
$$(\Xi_{k})_{R^{-1}}=\frac{1}{R^{k-1/2}} \Xi_{k}.$$
Since $U_=\sum_{k=1}^m \theta_k(R)\Xi_k$, we see that
$$U_{R^{-1}}=\sum_{k=1}^m \theta_k(R) (\Xi_{k})_{R^{-1}}=\sum_{k=1}^m \frac{\theta_k(R)}{R^{k-1/2}} \Xi_k,$$
and the conclusion of the claim follows from \eqref{normP1}.
\end{proof}

\subsection{Nonradiative solutions for the free wave equation in even space dimension}
In this subsection we prove Proposition \ref{P:Neven}. We assume $N\geq 4$ is even, and let $u$ be a solution of the free wave equation \eqref{LW}, with initial data $(u_0,u_1)$. We first assume that $u$ is radial. Denote by $(\hat{u}_0,\hat{u}_1)$ the Fourier transform of $(u_0,u_1)$ in~$\RR^N$. Then by \cite{CoKeSc14}, one has, for some constant $c_N>0$ 
\begin{multline} 
\label{ext_energy}
\sum_{\pm}\lim_{t\to\pm\infty} c_N\int_{|x|\geq |t|}  | \nabla_{t,x} u(t)|^2\,dx = 
\pi\int  (\rho^2|\hat{u}_0(\rho)|^2+|\hat{u}_1(\rho)|^2)\rho^{N-1}\, d\rho \\
 +(-1)^{\frac{N}{2}} \left( \int H(\rho^{\frac{N+1}{2}} \hat u_0)\, \rho^{\frac{N+1}{2}} \hat u_0\,d\rho - \int H( \rho^{\frac{N-1}{2}} \hat u_1)\rho^{\frac{N-1}{2}} \hat u_1 d\rho\right).
\end{multline}
Here $H$ is the Hankel  transform $H$ on the half-line $(0,\infty)$: 
\begin{equation*} 
(H\varphi)(\rho):=\int_0^\infty \frac{\varphi(\sigma)}{\rho+\sigma}\, d\sigma.
\end{equation*}
We claim that $H$ is bounded from $L^2$ to $L^2$ with operator norm equal to $\pi$, and that the operator norm is not attained, i.e
\begin{equation}
\label{boundH}
\forall f\in L^2((0,\infty))\setminus\{0\},\quad \|H f\|_{L^2}<\pi \|f\|_{L^2}. 
\end{equation} 
Assuming \eqref{boundH}, we obtain, as a consequence of \eqref{ext_energy} and Cauchy-Schwarz inequality, that if $(u_0,u_1)\neq (0,0)$, then 
$$\sum_{\pm}\lim_{t\to\pm\infty} c_N\int_{|x|\geq |t|}  | \nabla_{t,x} u(t)|^2\,dx>0.$$

The inequality \eqref{boundH} is classical. We give a proof for the sake of completeness. Let $\LLL$ be the Laplace transform: $\LLL f(s):=\int_0^{\infty} f(t)e^{-st}dt$. It is easy to check that $H=\LLL^2$. We are thus reduced to prove:
$$ \forall f\in L^2(0,\infty)\setminus \{0\}, \quad \left\|\LLL f\right\|_{L^2}<\sqrt{\pi}\|f\|_{L^2}.$$
Letting $g=\LLL f$, we obtain by Cauchy-Schwarz inequality:
\begin{multline}
\label{ineg_interm}
|g(s)|^2= \left( \int_0^{\infty} f(t)e^{-st/2}t^{1/4}t^{-1/4}dt \right)^2\leq \int_0^{\infty} |f(t)|^2e^{-st}t^{1/2}\,dt\int_0^{\infty} e^{-st}t^{-1/2}\,dt\\
\leq \sqrt{\pi}s^{-1/2} \int_0^\infty |f(t)|^2e^{-st}t^{1/2}\,dt,
\end{multline}
where we have used 
\begin{equation}
\label{calcul}
\int_0^{\infty} e^{-st} s^{1/2}t^{-1/2}\,dt=\int_0^{\infty}e^{-u}u^{-1/2}\,du=\sqrt{\pi}. 
\end{equation} 
Integrating \eqref{ineg_interm}, we obtain
\begin{equation}
\label{bound_Laplace}
\int_0^{\infty}|g(s)|^2\,ds \leq \sqrt{\pi}\int_0^{\infty} \int_0^{\infty} |f(t)|^2e^{-st}t^{1/2}s^{-1/2}\,dtds\leq \pi\int_0^{\infty}|f(t)|^2\,dt, 
\end{equation} 
where we have used \eqref{calcul} again. 

Assume now that equality holds in \eqref{bound_Laplace}. Then there is equality in \eqref{ineg_interm} for almost all $s>0$. This imposes that for almost all $s>0$, there is a real number $\lambda(s)$ such that for almost all $t>0$,
$$ f(t)t^{1/4}=\lambda(s)t^{-1/4}.$$
Thus $f(t)=\lambda/t^{1/2}$ for some $\lambda\in \RR$, a contradiction with the fact that $f$ is in $L^2$, unless $\lambda=0$ (and thus $f=0$ a.e.).

It remains to prove Proposition \ref{P:Neven} when $(u_0,u_1)$ is not assumed to be radial. We will reduce to the radial case by expanding the solution $u$ into spherical harmonics. Let $\left(\Phi_{k}\right)_{k\in \NN}$ be a Hilbert basis of spherical harmonics, and $(\nu_k)_{k\in \NN}$ the sequence of degrees of $\Phi_k$, that we can assume to be nondecreasing. Thus $\Phi_k$ is the restriction to $S^{N-1}$ of a homogeneous harmonic polynomial of degree $\nu_k\in \NN$ and 
$$-\Delta_{S^{N-1}} \Phi_k=\nu_k(\nu_k+N-2)\Phi_k.$$
We let, for $t\in \RR$ and $r>0$,
 $$u_k(t,r)= \int_{S^{N-1}} \Phi_k(\theta)u(t,r\theta)\,d\sigma(\theta).$$
We have $\vec{u}_k\in C^0(\RR,\HHH)$ and
 $\partial_t^2u_k-\Delta u_k+\frac{\nu_k(\nu_k+N-2)}{r^2}u_k=0.$
Letting $v_k=r^{-\nu_k} u_{k}$, we obtain:
\begin{equation*}
\partial_t^2v_k-\Delta_D v_k=0,
\end{equation*} 
with initial data
\begin{equation*}
\vec{v}_{k\restriction t=0}=(v_{0k},v_{1k}):=r^{-\nu_k}\vec{u}_k(0), 
\end{equation*} 
where $\Delta_{D_k}=\partial_r^2+\frac{D_k-1}{r}\partial_r$
is the radial part of the Laplace operator in dimension $D_k=N+2\nu_k$. We thus identify $\vec{v}_k(t)$, for $t\in \RR$, with a radial function on $\RR^{D_k}$. Noting that $(v_{0k},v_{1k})$ is in $\dot{H}^1\left(\RR^{D_k}\right)\times L^2\left( \RR^{D_k} \right)$, we obtain, 
from the radial case treated above, that if $v_k$ is not the zero solution, we have
$$\sum_{\pm}\lim_{t\to+\infty} \int_{|t|}^{+\infty} |\partial_{t,r}v_k(r)|^2r^{D_k-1}dr>0.$$
A direct computation, using that 
$\lim_{t\to \pm\infty} \int_{\RR^N}\frac{1}{|x|^2}|u(t,x)|^2\,dx=0,$
shows that
$$\sum_{\pm}\lim_{t\to+\infty} \int_{|t|}^{+\infty} |\partial_{t,r}u_k(r)|^2r^{N-1}dr>0.$$
This implies the conclusion of Proposition \ref{P:Neven} in this case also.

\subsection{Asymptotic behaviour for non-radiative solutions of the nonlinear wave equation}
\label{SS:asympt}
Fix now $N\geq 3$ odd. In this subsection we prove Theorems \ref{T:maintheo} and \ref{T:maintheo2} as a consequence of the following results.
\begin{prop} 
\label{P:asymptotic_NR}
Let $u$ be a weakly non-radiative solution of \eqref{NLW}. Then there exists $k_0\in\llbracket 1,m\rrbracket$ and $\ell\in \RR$, with $\ell\neq 0$ if $k_0<m$ and such that
\begin{equation}
 \label{asymptotic_NR}
\left\|(u_0,u_1)-\ell \,\Xi_{k_0}\right\|_{\HHH(R)}\lesssim \max\left(\frac{1}{R^{(k_0-\frac{1}2)\frac{N+2}{N-2}}},\frac{1}{R^{k_0+\frac 12}}\right),
\end{equation} 
where the implicit constant depends on $u$.
Furthermore if $k_0=m$ then for almost every large $r$,
$$(u_0,u_1)(r)=\begin{cases}
0 & \text{ if }\ell=0\\
  \mathrm{sign}(\ell)\big(W_{(\lambda)}(r),0\big)& \text{ if }\ell \neq 0,             
              \end{cases}$$
where $\lambda$ is defined by
$\lambda^{\frac N2-1}\Big(N(N-2)\Big)^{\frac N2-1}=|\ell|$.
 \end{prop}
\begin{remark}
\label{R:asymptotic_NR}
 We can give an estimate on the implicit constant in \eqref{asymptotic_NR}. Let $\eps>0$ be a small parameter, and let $R_{\eps}$ such that $u$ is well-defined for $|x|>R_{\eps}+|t|$ and 
\begin{equation}
 \label{defReps}
 \|(u_0,u_1)\|_{\HHH(R_{\eps})}\leq \eps.
\end{equation} 
 Then
\begin{equation}
\label{asymptotic_NRbis}
\left\|(u_0,u_1)-\ell \,\Xi_{k_0}\right\|_{\HHH(R)}\leq 
C\,\eps\, \max\left\{\left(\frac{R_{\eps}}{R}\right)^{(k_0-\frac{1}2)\frac{N+2}{N-2}},\left(\frac{R_{\eps}}{R}\right)^{k_0+\frac 12}\right\} ,
\end{equation} 
where the constant $C$ depends only on $N$. 
\end{remark}
\begin{prop}
 \label{P:l_no_t}
Let $u$ be a weakly non-radiative solution of \eqref{NLW} such that the essential support of $(u_0,u_1)$ is not compact. For $T\in \RR$, let $k_0(T)$, $\ell(T)$ the parameters defined in Proposition \ref{P:asymptotic_NR} for the solution $u(T+\cdot)$. Then $k_0$ and $\ell$ are independent of $T$. 
\end{prop}
\begin{remark}
 The assumption on the essential support of $(u_0,u_1)$ is satisfied if and only if $k_0<m$ or $k_0=m$ and $(u_0,u_1)(r)=(\pm W_{(\lambda)}(r),0)$ for large $r$.
\end{remark}

Finally, we note that with the stronger assumption that $u$ is a nonradiative solution, we can improve the conclusion of Proposition \ref{P:asymptotic_NR} and obtain the partial uniqueness result:
\begin{corol}
 \label{C:uniqueness_NR}
Let $u$ be a nonradiative solution of \eqref{NLW} at $t=0$, that is a solution defined on $\{|x|\geq |t|\}$ such that
$$ \sum_{\pm \infty}\lim_{t\to\pm\infty}\int_{\{|x|>|t|\}} |\nabla_{t,x}u(t,x)|^2\,dx=0.$$
Let $k_0$ be as in Proposition \ref{P:asymptotic_NR} and assume that $k_0=m$. Then $(u_0,u_1)=(0,0)$ or there exists $\lambda>0$ and a sign $\pm$ such that $(u_0,u_1)=(\pm W_{(\lambda)},0)$. 
\end{corol}
Let us postpone the proof of the propositions and prove Corollary \ref{C:uniqueness_NR}. We will need the following proposition, valid for all $N\geq 3$ (see \cite[Subsection 3.6]{DuKeMe19Pa} for the proof):
\begin{prop}
\label{P:support}
Let $V(t,r)$ be a continuous, real-valued potential defined on $\{|x|>R+|t|\}$ for some $R$ and that satisfies 
\begin{equation*}
 r>R+|t|\Longrightarrow |V(t,r)|\leq \frac{C}{r^2}
\end{equation*} 
for some constant $C$. 
Let $h$ be a radial solution of 
$$ \partial^2_th-\Delta h+V h=0,\quad |x|>R+|t|.$$
Let $\rho_0>R$ and assume that the support of $(h_0,h_1)=\vec{h}(0)$ is included in $\{r\leq \rho_0\}$. Then there exist $\eps>0$ such that for all $\rho\in \big(\rho_0-\eps,\rho_0\big)$, the following holds for all $t\geq 0$ or for all $t\leq 0$:
\begin{equation*}
 \int_{\rho+|t|}^{+\infty} (\partial_{t,r}h(t,r))^2r^{N-1}\,dr\geq \frac{1}{8}\int_{\rho}^{+\infty} (\partial_{t,r}h(0,r))^2r^{N-1}\,dr.
\end{equation*} 
\end{prop}
\begin{proof}[Proof of the corollary]
We argue by contradiction. 
By the assumption on $u$
\begin{equation}
 \label{u_H_bounded}
 \sup_{t\in \RR} \|u\|_{\HHH(|t|)}<\infty,
\end{equation} 
By the radial Sobolev embedding theorem,
\begin{equation}
\label{toprove_u}
\forall (t,r),\quad r>|t|\Longrightarrow 
|u(t,r)|^{\frac{4}{N-2}}\leq \frac{C}{r^2}. 
\end{equation} 
First assume that $k_0=m$ and $\ell=0$. Then by Proposition \ref{P:asymptotic_NR}, $\vec{u}(t)$ is compactly supported for all $t$. Since $u$ satisfies 
$$\partial_t^2u-\Delta u=|u|^{\frac{4}{N-2}}u,$$
we directly obtain a contradiction from Proposition \ref{P:support} with $V=|u|^{\frac{4}{N-2}}$  and the fact that $u$ is nonradiative.

The proof is very similar if $\ell\neq 0$. In this case, we know from Proposition \ref{P:asymptotic_NR} that there exists $\lambda>0$ and $\iota \in \{\pm 1\}$ such that 
$$(u_0,u_1)-(\iota W_{(\lambda)},0)$$
is compactly supported. Assume to fix the ideas that $\iota=1$ and $\lambda=1$, and let $h=u-W$. 
Then $h$ satisfies
$$\partial_t^2h-\Delta h=F(W+h)-F(W)=Vh,$$
where 
$$V=\frac{F(W+h)-F(W)}{h}$$
is such that
$$\left|V\right|\leq W^{\frac{4}{N-2}}+|h|^{\frac{4}{N-2}}.$$
Using as before the radial Sobolev embedding and \eqref{toprove_u}, we see that there exists a constant $C>0$ such that
$$r\geq |t|\Longrightarrow\left|V(t,r)\right| \leq \frac{C}{r^2}.$$
Proposition \ref{P:support} and the fact that $u$ (and thus $h$) is nonradiative yield a contradiction.
\end{proof}

Before proving the propositions, we will state a simple result on sequences with geometric growth:
\begin{claim}
 \label{Cl:sequences}
 Let $q,r\in (0,1)$, $c_0\geq 0$, $\beta >1$. Then there exist a small constant $\eps$ and a large constant $C$ (both depending only on $q$, $r$, $c_0$ and $\beta$), with the following property.
 Let $(\mu_n)_n$ be a sequence and $\nu_0\in [0,\eps]$ such that 
 \begin{gather}
  \label{AZ10}
  \forall n,\quad 0\leq \mu_n\leq \eps\\
  \label{AZ11}
  \forall n,\quad \mu_{n+1}\leq q\mu_n+c_0\mu_n^{\beta}+\nu_0r^n.
 \end{gather}
Then if $q\neq r$,
\begin{equation}
\label{AZ12}
\forall n\geq 0, \quad \mu_n\leq C (\mu_0+\nu_0)\max\{ q^n,r^n\},
\end{equation} 
and if $q=r$,
\begin{equation}
 \label{AZ13}
 \mu_n\leq C\left( \mu_0+\nu_0(1+n) \right)r^n.
\end{equation} 
\end{claim}
The claim is proved in appendix \ref{A:sequences}
\begin{proof}[Proof of Proposition \ref{P:asymptotic_NR}]
 \setcounter{step}{0}
 \begin{step}
\label{St:asy_prelim}
 In all the proof, we fix a small $\eps>0$ (smallness is independent of $u$) and a $R_{\eps}>0$ such that $u(t,r)$ is defined for $r>R_{\eps}+|t|$ and \eqref{defReps} is satisfied. Note in particular that it implies, by the small data well-posedness theory (Proposition \ref{P:LWP_cone}) that for all $R\geq R_{\eps}$, 
 \begin{equation}
 \label{MZ_small_data}
  \|u-u_L\|_{S(\{|x|>R+|t|\})}+\|u-u_L\|_{W(\{|x|>R+|t|\})}\lesssim \|(u_0,u_1)\|^{\frac{N+2}{N-2}}_{\HHH(R)},
 \end{equation}
 where $u_L(t)=S_L(t)(u_0,u_1)$.
 For $R>R_{\eps}$, define:
 $$v_R(t)=S_L(t)\Pi_{P(R)}(u_0,u_1),\quad w_R(t)=u(t)-v_R(t),$$
so that 
\begin{equation*}
\partial_t^2w_R-\Delta w_R=|u|^{\frac{4}{N-2}}u,\quad \pi_{P(R)}(w_{R0},w_{R1})=(0,0),
\end{equation*} 
where $(w_{R0},w_{R1})=\vec{w}_R(0)$. As a consequence, for all $T$, denoting by $w_{RF}$ the solution of the free wave equation with initial data $(w_{R0},w_{R1})$ at $t=0$, 
$$\left\|\vec{w}_R(T)-\vec{w}_{RF}(T)\right\|_{\HHH(R+|T|)}\lesssim \left\||u|^{\frac{4}{N-2}}u\right\|_{W'(\{|x|>R+|T|\})}\lesssim \|(u_0,u_1)\|^{\frac{N+2}{N-2}}_{\HHH(R)},$$
where we have used \eqref{MZ_small_data} and the fractional chain rule \eqref{fractional}.
This implies 
$$\sum_{\pm}\lim_{t\to\pm\infty} \int_{|x|\geq R+|t|} |\nabla_{t,x}w_{RF}(t,x)|^2\,dx\lesssim \|(u_0,u_1)\|_{\HHH(R)}^{\frac{2(N+2)}{N-2}},$$
and thus by \cite{KeLaLiSc15} (see \eqref{linear_exterior}),
$$ \left\|\pi_{P(R)^{\bot}}(u_0,u_1)\right\|_{\HHH(R)}=\|(w_{R0},w_{R1})\|_{\HHH(R)}\lesssim \left\|(u_0,u_1)\right\|_{\HHH(R)}^{\frac{N+2}{N-2}}.$$
 Using that $\|(u_0,u_1)\|_{\HHH(R)}$ is small, we deduce
 \begin{equation}
  \label{MZ20}
  \left\|\pi_{P(R)^{\bot}}(u_0,u_1)\right\|_{\HHH(R)}\lesssim \left\|\pi_{P(R)}(u_0,u_1)\right\|_{\HHH(R)}^{\frac{N+2}{N-2}}.
 \end{equation} 
 We denote by $\theta_k(R)$ the coordinates of $\pi_{P(R)}(u_0,u_1)$ in the basis $(\Xi_1,\ldots,\Xi_m)$, i.e.
 $$(u_0,u_1)=(w_{R0},w_{R1})+\sum_{k=1}^m \theta_k(R)\Xi_k.$$
 By \eqref{MZ20} and Claim \ref{Cl:coord},
 \begin{equation}
  \label{MZ30}
  \left\|(w_{R0},w_{R1})\right\|_{\HHH(R)}\lesssim \left( \sum_{k=1}^{m}\frac{|\theta_k(R)|}{R^{k-\frac 12}} \right)^{\frac{N+2}{N-2}}.
 \end{equation} 
 We next consider $R'$ such that $R_{\eps}\leq R\leq R'\leq 2R$. For $r\geq R'$, writing $(u_0,u_1)(r)$ in two different ways, we obtain
 $$ (w_{R0},w_{R1})(r)+\sum_{k=1}^m \theta_k(R)\Xi_k(r)=(w_{R'0},w_{R'1})(r)+\sum_{k=1}^m \theta_k(R')\Xi_k(r),$$
 and thus
 $$\sum_{k=1}^m (\theta_k(R)-\theta_k(R'))\Xi_k(r)=-(w_{R0},w_{R1})+(w_{R'0},w_{R'1}).$$
 Using \eqref{MZ30} and Claim \ref{Cl:coord}, we deduce
 $$\sum_{k=1}^m |\theta_{k}(R)-\theta_{k}(R')|\frac{1}{(R')^{k-\frac 12}}\lesssim \left( \sum_{k=1}^m \frac{|\theta_k(R)|}{R^{k-\frac 12}}+ \sum_{k=1}^m \frac{|\theta_k(R')|}{{R'}^{k-\frac 12}}\right)^{\frac{N+2}{N-2}}.$$
 Since $R\leq R'\leq 2R$, we can replace $\frac{1}{(R')^{k-\frac 12}}$ by $\frac{1}{R^{k-\frac 12}}$ in the inequality above. Using the smallness of $\sum_{k=1}^{m}\frac{|\theta_k(R)|}{R^{k-\frac 12}}$, we obtain
 \begin{equation*}
  \sum_{k=1}^{m}\frac{|\theta_k(R')|}{R^{k-\frac 12}}\lesssim\sum_{k=1}^{m}\frac{|\theta_k(R)|}{R^{k-\frac 12}},
 \end{equation*} 
 and finally, for all $R,R'$ such that $R_{\eps}\leq R\leq R'\leq 2R$,
\begin{equation}
\label{MZ31}
\sum_{k=1}^m |\theta_{k}(R)-\theta_{k}(R')|\frac{1}{R^{k-\frac 12}}\lesssim \left( \sum_{k=1}^m \frac{|\theta_k(R)|}{R^{k-\frac 12}}\right)^{\frac{N+2}{N-2}}.
\end{equation} 
In the two next steps, we will prove, using repeatedly \eqref{MZ31}, that there exist $k_0\in \llbracket 1,m\rrbracket$ and $\ell_{k_0}\in \RR$ such that for all $R\geq R_{\eps}$
\begin{align}
 \label{MZ40}
 \frac{\left|\theta_{k_0}(R)-\ell_{k_0}\right|}{R_{\eps}^{k_0-\frac 12}}&\lesssim \eps^{\frac{N+2}{N-2}} \left( \frac{R_{\eps}}{R} \right)^{\frac{4(k_0-\frac 12)}{N-2}}\\
\label{MZ41}
 \forall k\in \llbracket 1, k_0-1\rrbracket,\quad \frac{\left|\theta_{k}(R)\right|}{R^{k-\frac 12}}&\lesssim \eps\left( \frac{R_{\eps}}{R} \right)^{\left(\frac{N+2}{N-2}\right)(k_0-\frac 12)}\\
\label{MZ42}
 \forall k\in \llbracket k_0+1,m\rrbracket,\quad \frac{\left|\theta_{k}(R)\right|}{R^{k-\frac 12}}&\lesssim \eps\left( \frac{R_{\eps}}{R} \right)^{\min\left\{k_0+\frac 12,\left(\frac{N+2}{N-2}\right)(k_0-\frac 12)\right\}},
 \end{align}
with $\ell_{k_0}\neq 0$ or $k_0=m$. 

We first check that \eqref{MZ40}, \eqref{MZ41} and \eqref{MZ42} imply that the desired estimate \eqref{asymptotic_NRbis} holds. From \eqref{MZ40} at $R=R_{\eps}$, we have $|\ell_{k_0}|\lesssim \eps R_{\eps}^{k_0-\frac 12}$ and thus
$$\frac{1}{R^{k_0-\frac 12}} |\theta_{k_0}(R)|\lesssim |\ell_{k_0}| \frac{1}{R^{k_0-\frac 12}} +\eps^{\frac{N+2}{N-2}} \left( \frac{R_{\eps}}{R} \right)^{\frac{(k_0-\frac 12)N+2}{N-2}}\lesssim\eps\left(\frac{R_{\eps}}{R}\right)^{k_0-\frac 12},$$
for $R\geq R_{\eps}$. Combining with \eqref{MZ41} and \eqref{MZ42}, and the estimate \eqref{MZ30} we obtain
$$ \left\|\left( w_{R0},w_{R1}) \right)\right\|_{\HHH(R)}\lesssim \eps\left( \frac{R_{\eps}}{R} \right)^{\frac{N+2}{N-2}(k_0-\frac 12)},$$
and thus 
\begin{multline*}
\|(u_0,u_1)-\ell_{k_0}\Xi_{k_0}\|_{\HHH(R)}\approx \left\|\pi_{P(R)}(u_0,u_1) -\ell_{k_0}\Xi_{k_0}\right\|_{\HHH(R)}+\|(w_{R0},w_{R1})\|_{\HHH(R)}\\
\lesssim \eps \left( \frac{R_{\eps}}{R} \right)^{\min\left\{k_0+\frac{1}{2},\frac{N+2}{N-2}\left(k_0-\frac 12\right)\right\} },
\end{multline*}
which yields \eqref{asymptotic_NRbis} with $\ell_{k_0}=\ell$. 

We will prove \eqref{MZ40}, \eqref{MZ41} and \eqref{MZ42} in Steps \ref{St:k01} and \ref{St:k0heredity} by induction on $k_0$, stopping the induction process when we have found $k_0$ such that \eqref{MZ40}, \eqref{MZ41} and \eqref{MZ42} hold with $k_0=m$ or $\ell_{k_0}\neq 0$. 

In Steps \ref{St:l0} and \ref{St:k0m} we will conclude the proof of Proposition \ref{P:asymptotic_NR}, proving that in the case where $k_0=m$, $(u_0,u_1)(r)$ is for large $r$ equal to $(0,0)$ or $(\pm W_{(\lambda)}(r),0)$ for some $\lambda>0$.
 \end{step}
 \begin{step}[The case $k_0=1$]
 \label{St:k01}
We prove here that \eqref{MZ40} and \eqref{MZ42} hold when $k_0=1$ for some $\ell_1\in \RR$.  (Note that \eqref{MZ41} is trivial in this case.)
For $1\leq k\leq m$, we define
\begin{equation}
 \label{def_Ak}
 A_k(R)=\sum_{j=k}^{m} \frac{|\theta_{j}(R)|}{R^{j-\frac 12}},
\end{equation} 
and note that $|A_k(R)|\lesssim \|(u_0,u_1)\|_{\HHH(R)}\lesssim \eps$ if $R\geq R_{\eps}$. We first prove
\begin{equation}
  \label{MZ52}
  R\geq R_{\eps}\Longrightarrow A_1(R)\lesssim \eps\left( \frac{R_{\eps}}{R} \right)^{\frac 12}.
 \end{equation} 
 By  
 \eqref{MZ31} with $R'=2R$, we have
 $$A_1(2R)=\sum_{k=1}^m \frac{|\theta_{k}(2R)|}{(2R)^{k-\frac 12}} \leq \frac{1}{2^{\frac 12}} \sum_{k=1}^m\frac{|\theta_k(2R)|}{R^{k-\frac 12}} \leq \frac{A_1(R)}{2^{\frac 12}}+CA_1(R)^{\frac{N+2}{N-2}}.$$
 Thus, for all $n\geq 0$, 
 \begin{equation}
  \label{MZ50}
  A_1\left( 2^{n+1}R_{\eps} \right)\leq \frac{A_1(2^nR_{\eps})}{2^{\frac 12}}+CA_1(2^nR_{\eps})^{\frac{N+2}{N-2}}.
 \end{equation} 
 Using Claim \ref{Cl:sequences}, we deduce
 $$ A_1(2^nR_{\eps})\lesssim \frac{A_1(R_{\eps})}{2^{\frac{n}{2}}}\lesssim \frac{\eps}{2^{\frac{n}{2}}}.$$
 This is \eqref{MZ52} when $R$ is of the form $2^nR_{\eps}$ for some $n\in \NN$.
 Since by \eqref{MZ31}, if $R_{\eps}\leq R\leq R'\leq 2R$,
 $$ A_1(R')\lesssim A_1(R),$$
 we deduce \eqref{MZ52} for all $R\geq R_{\eps}$.
 
 We next prove that there exists $\ell_{1}$ such that \eqref{MZ40} holds. Combining \eqref{MZ31} with the bound \eqref{MZ52} on $A_1(R)$, we have, for $R_{\eps}\leq R\leq R'\leq 2R$.
 \begin{equation}
 \label{MZdiff}
 \frac{1}{R^{\frac 12}}|\theta_1(R)-\theta_1(R')|\lesssim \left( A_1(R) \right)^{\frac{N+2}{N-2}}\lesssim \eps^{\frac{N+2}{N-2}}\left( \frac{R_{\eps}}{R} \right)^{\frac{N+2}{2(N-2)}}.
 \end{equation} 
 Hence for all $n\geq 0$,
 $$\left|\theta_{1}(2^nR_{\eps})-\theta_1(2^{n+1}R_{\eps})\right|\lesssim \eps^{\frac{N+2}{N-2}} R_{\eps}^{\frac 12} 2^{-\frac{2n}{N-2}}.$$
 As a consequence, $\sum_n \left|\theta_1(2^nR_{\eps})-\theta_1(2^{n+1}R_{\eps})\right|$ converges, and there exists $\ell_1\in \RR$ such that 
 $$|\theta_{1}(2^n(R_{\eps}))-\ell_1|\lesssim \eps^{\frac{N+2}{N-2}} R_{\eps}^{\frac 12} 2^{-\frac{2n}{N-2}}.$$
 This is \eqref{MZ40} in the case $k_0=1$, where $R$ is restricted to the values  $2^nR_{\eps}$ with $n\in \NN$. The inequality \eqref{MZ40} for general $R$ follows, using \eqref{MZdiff}. We note that \eqref{MZ40} with $R=R_{\eps}$ implies  $|\ell|\lesssim \eps R_{\eps}^{\frac 12}$ and thus
 \begin{equation}
  \label{MZboundtheta1}
  R\geq R_{\eps}\Longrightarrow|\theta_1(R)|\lesssim \eps R_{\eps}^{\frac 12}.
 \end{equation} 
 We next prove \eqref{MZ42}. We must bound $A_2(R)$, defined in \eqref{def_Ak}. By \eqref{MZ31},
 $$ \sum_{k=2}^m \frac{|\theta_k(2R)|}{R^{k-\frac 12}}\leq
 \sum_{k=2}^m \frac{|\theta_k(R)|}{R^{k-\frac 12}}+C\left( \sum_{k=2}^m \frac{|\theta_k(R)|}{R^{k-\frac 12}} +\frac{|\theta_1(R)|}{R^{\frac 12}} \right)^{\frac{N+2}{N-2}}.$$
 This yields, using \eqref{MZboundtheta1}.
 $$A_2(2R)\leq \frac{1}{2^{\frac 32}}A_2(R) +C\left( A_2(R) \right)^{\frac{N+2}{N-2}}+C\left( \frac{R_{\eps}}{R} \right)^{\frac{N+2}{2(N-2)}} \eps^{\frac{N+2}{N-2}}.$$
 Hence, for $n\geq 0$,
 $$A_2\left(2^{n+1}R_{\eps}\right)\leq \frac{1}{2^{\frac 32}}A_2\left(2^nR_{\eps}\right) +C\left( A_2(2^nR_{\eps}) \right)^{\frac{N+2}{N-2}}+C 2^{-\frac{n(N+2)}{2(N-2)}} \eps^{\frac{N+2}{N-2}}.$$
 Using Claim \ref{Cl:sequences}, we deduce
 $$ A_2(2^{n}R_{\eps})\lesssim \max \left\{ \frac{1}{2^{\frac 32n}}, \frac{1}{2^{\frac{N+2}{2(N-2)}n}}\right\}\eps.$$
 In other terms, the following inequality holds when $R=2^nR_{\eps}$ for some integer $n\geq 0$:
 $$A_2(R)\lesssim \max\left\{ \left(\frac{R_{\eps}}{R}\right)^{\frac 32}, \left(\frac{R_{\eps}}{R}\right)^{\frac{N+2}{2(N-2)}}\right\}\eps.$$
 Arguing as before, we deduce that it holds for all $R\geq R_{\eps}$, and thus that \eqref{MZ42} holds with $k_0=1$.
 \end{step}
\begin{step}[Heredity]
\label{St:k0heredity}
 If the limit $\ell_1$ defined in Step \ref{St:k01} is not zero, then we are done. If $\ell_1=0$, we continue the same process. More precisely, we prove that if $2\leq  k_0\leq m$ and \eqref{MZ40}, \eqref{MZ41} and \eqref{MZ42} holds at rank $k_0-1$, and $\ell_{k_0-1}=0$  in \eqref{MZ40}, then \eqref{MZ40} (for some $\ell_{k_0}\in \RR$), \eqref{MZ41} and \eqref{MZ42} hold at rank $k_0$. We thus assume, that for all $R\geq R_{\eps}$, 
 \begin{align}
  \label{MZ80}  
  \frac{|\theta_{k}(R)|}{R^{k-\frac 12}}&\lesssim \left( \frac{R_{\eps}}{R} \right)^{\frac{N+2}{N-2}\left( k_0-\frac 32 \right)}\eps, \quad 1\leq k\leq k_0-1\\
  \label{MZ81}
  \frac{|\theta_k(R)|}{R^{k-\frac 12}}&\lesssim C\eps \max \left\{\left( \frac{R_{\eps}}{R} \right)^{k_0-\frac 12},\left( \frac{R_{\eps}}{R} \right)^{\frac{N+2}{N-2}\left( k_0-\frac{3}{2} \right)}\right\},\quad k_0\leq k\leq m.
 \end{align}
Note that \eqref{MZ80} when $k=k_0-1$ follows from \eqref{MZ40} at rank $k_0-1$ when $\ell_{k_0-1}=0$. We prove \eqref{MZ41} by induction, proving that an estimate of the form \eqref{MZ80} improves automatically. Assume
\begin{equation}
 \label{MZ82}
 \forall k\in \llbracket 1,k_0-1\rrbracket,\quad \frac{|\theta_k(R)|}{R^{k-\frac 12}}\lesssim \left( \frac{R_{\eps}}{R} \right)^{\alpha}\eps
\end{equation} 
for some $\alpha$ such that $\alpha >k_0-\frac 32$. Note that \eqref{MZ80} is \eqref{MZ82} with $\alpha=\frac{N+2}{N-2}\left(k_0-\frac{3}{2}\right)$. We consider $A_{k_0}(R)$ defined by \eqref{def_Ak}. By \eqref{MZ31} and \eqref{MZ82},
$$A_{k_0}\left( 2^{n+1}R_{\eps} \right)\leq \frac{1}{2^{k_0-\frac 12}} A_{k_0}\left( 2^nR_{\eps} \right)+C\Big(A_{k_0}\left( 2^nR_{\eps} \right)\Big)^{\frac{N+2}{N-2}} +C \eps^{\frac{N+2}{N-2}} \left(\frac{1}{2^n}\right)^{\alpha\frac{N+2}{N-2}}.$$
By Claim \ref{Cl:sequences}, we have
\begin{equation}
\label{MZ83}
A_{k_0}\left( 2^nR_{\eps} \right)\leq C\eps \max \left\{\left( \frac{1}{2^n} \right)^{k_0-\frac 12},\left( \frac{1}{2^n} \right)^{\alpha \left( \frac{N+2}{N-2} \right)} \right\}, 
\end{equation} 
if $k_0-\frac 12=\alpha \left( \frac{N+2}{N-2} \right)$ there is an additional $(n+1)$ factor. We will assume, taking a slightly smaller $\alpha$ if necessary, that we are not in this case.
This yields, for $R\geq R_{\eps}$,
\begin{equation}
 \label{MZ90}
 A_{k_0}(R)\lesssim C\eps \max \left\{\left( \frac{R_{\eps}}{R} \right)^{k_0-\frac 12},\left( \frac{R_{\eps}}{R} \right)^{\alpha \left( \frac{N+2}{N-2} \right)} \right\}.
\end{equation} 
 If $1\leq k\leq k_0-1$, we can combine the estimates \eqref{MZ31}, \eqref{MZ82} and \eqref{MZ83} to obtain
$$\frac{1}{(2^nR_{\eps})^{k-\frac{1}{2}}}\left|\theta_k(2^nR_{\eps})-\theta_k(2^{n+1}R_{\eps})\right|\lesssim \eps^{\frac{N+2}{N-2}}\left( \max\left\{\left( \frac{1}{2^n} \right)^{k_0-\frac 12},\left( \frac{1}{2^n} \right)^{\alpha}\right\}  \right)^{\frac{N+2}{N-2}}.$$
Since by our assumption on $\alpha$, $k-\frac{1}{2}\leq k_0-\frac{3}{2}<\alpha\frac{N+2}{N-2}$, we obtain
\begin{multline*}
\sum_{p\geq n}\left|\theta_k(2^pR_{\eps})-\theta_k(2^{p+1}R_{\eps})\right|\\
\lesssim \eps^{\frac{N+2}{N-2}}\sum_{p\geq n}(2^pR_{\eps})^{k-\frac{1}{2}}\left( \max\left\{\left( \frac{1}{2^p} \right)^{k_0-\frac 12},\left( \frac{1}{2^p} \right)^{\alpha}\right\}  \right)^{\frac{N+2}{N-2}}.
\end{multline*}
By \eqref{MZ82} and since $\alpha>k_0-\frac 32$, we have $\lim_{R\to\infty} \theta_k(R)=0$. We can thus deduce from the preceding inequality:
$$\frac{1}{(2^nR_{\eps})^{k-\frac{1}{2}}}\left|\theta_k(2^nR_{\eps})\right|\lesssim \eps^{\frac{N+2}{N-2}}\left( \max\left\{\left( \frac{1}{2^n} \right)^{k_0-\frac 12},\left( \frac{1}{2^n} \right)^{\alpha}\right\}  \right)^{\frac{N+2}{N-2}},$$
and, using the same argument as before, we can extend this inequality and obtain:
$$ \forall R\geq R_{\eps},\quad  \frac{|\theta_k(R)|}{R^{k-\frac 12}} \lesssim \eps \left( \frac{R_{\eps}}{R} \right)^{\alpha'},$$
where $\alpha'=\min \left( (k_0-\frac 12)\frac{N+2}{N-2},\alpha \frac{N+2}{N-2} \right)$. This is \eqref{MZ82} with $\alpha$ replaced with $\alpha'$. Iterating, we see that \eqref{MZ82} holds with $\alpha=( k_0-\frac 12)\frac{N+2}{N-2}$, which is exactly \eqref{MZ41}. Note that the above proof also yields that \eqref{MZ90} holds with this value of $\alpha$, that is\begin{equation}
\label{MZ100}
A_{k_0}(R)\lesssim \eps \left( \frac{R_{\eps}}{R} \right)^{k_0-\frac{1}{2}}.
\end{equation} 
We next prove that \eqref{MZ40} holds for some $\ell_{k_0}\in \RR$. By \eqref{MZ31}, \eqref{MZ41} and \eqref{MZ100},
$$ \left|\theta_{k_0}(2^nR_{\eps})-\theta_{k_0}(2^{n+1}R_{\eps})\right|\lesssim \frac{\eps^{\frac{N+2}{N-2}} R_{\eps}^{k_0-\frac 12}}{2^{\frac{4n(k_0-\frac 12)}{N-2}}}.$$
This proves that $\theta_{k_0}(2^nR_{\eps})$ has a limit $\ell_{k_0}$ as $n\to\infty$, and that 
$$ \left|\theta_{k_0}(2^nR_{\eps})-\ell_{k_0}\right|\lesssim \frac{\eps^{\frac{N+2}{N-2}} R_{\eps}^{k_0-\frac 12}}{2^{\frac{4n\left(k_0-\frac 12\right)}{N-2}}}.$$
As usual, we can deduce that $\theta_{k_0}(R)$ has a limit $\ell_{k_0}$ as $R\to\infty$, and that \eqref{MZ40} is satisfied. 

If $k_0=m$, the assertion \eqref{MZ42} is empty and we are done. If $k_0\leq m-1$, we must prove that \eqref{MZ42} holds. Using \eqref{MZ31}, \eqref{MZ40}, \eqref{MZ41} and \eqref{MZ100}, we see that 
\begin{equation}
 \label{MZ110}
 A_{k_0+1}\left( 2^{n+1}R_{\eps} \right)\leq \frac{A_{k_0+1}(2^nR_{\eps})}{2^{k_0+\frac{1}{2}}} +C\left( A_{k_0+1}(2^nR_{\eps})\right)^{\frac{N+2}{N-2}} +C\left( \frac{\eps}{2^{n(k_0-\frac{1}{2})}} \right)^{\frac{N+2}{N-2}}.
\end{equation} 
Using Claim \ref{Cl:sequences}, we see that \eqref{MZ110} implies 
$$|A_{k_0+1}(R)|\lesssim \max\left\{\left( \frac{R_{\eps}}{R} \right) ^{k_0+\frac 12},\left( \frac{R_{\eps}}{R} \right)^{\frac{N+2}{N-2}(k_0-\frac 12)}\right\}\eps,$$
for $R=2^nR_\eps$, $n\in \NN$, and then, with the same argument as above, for all $R\geq R_{\eps}$. This is exactly \eqref{MZ42}, which concludes this step.
\end{step}
\begin{step}
\label{St:l0}
 In this step we assume that $k_0=m$ and that $\ell_m=0$, and we prove that $(u_0,u_1)(r)=(0,0)$ for almost every large $r$. We first prove that for all $\alpha>0$, there exists $C_{\alpha}>0$ such that 
 \begin{equation}
  \label{MZ130} 
  \forall R\geq R_{\eps}, \quad A_1(R) \leq C_{\alpha} \left( \frac{R_{\eps}}{R} \right)^{\alpha}.
 \end{equation} 
 Indeed, by \eqref{MZ40} and \eqref{MZ41} with $k_0=m$ and $\ell_{k_0}=0$, \eqref{MZ130} holds with $\alpha=\frac{(m-\frac 12)(N+2)}{N-2}$. Moreover, if \eqref{MZ130} holds for some $\alpha\geq m-\frac 12$, then by \eqref{MZ31},
 $$\sum_{k=1}^m \frac{|\theta_k(R)-\theta_k(2R)|}{R^{k-\frac 12}} \lesssim \left( \frac{R_{\eps}}{R} \right)^{\frac{N+2}{N-2}\alpha}.$$
 Thus, for all $k\in \llbracket 1,m\rrbracket$, for all $R\geq R_{\eps}$, 
 $$ \left|\theta_{k}(2^nR)-\theta_{k}(2^{n+1}R)\right|\lesssim \frac{R^{k-\frac 12}} {2^{n\left( \frac{(N+2)\alpha}{N-2}-k+\frac 12\right)} }\left(\frac{R_{\eps}}{R} \right)^{\frac{N+2}{N-2}\alpha}.$$
Since $\alpha\geq m-\frac{1}{2}$, we have $\frac{N+2}{N-2}\alpha>k-\frac{1}{2}$. Summing up the inequality above over all $n\geq 1$, and using again that $\lim_{R\to\infty}\theta_k(R)=0$, we obtain
$$ \frac{1}{R^{k-\frac 12}} |\theta_k(R)|\lesssim \left( \frac{R_{\eps}}{R} \right)^{\frac{N+2}{N-2}\alpha}.$$
Thus \eqref{MZ130} holds with $\alpha$ replaced by $\frac{N+2}{N-2}\alpha$. As a conclusion, \eqref{MZ130} holds for all $\alpha>0$.

Combining \eqref{MZ130} with \eqref{MZ31}, we obtain that for all $R\geq R_{\eps}$,
$$\sum_{k=1}^m |\theta_k(R)-\theta_k(2R)|\leq C\,C_{\alpha}^{\frac{4}{N-2}} R^{m-1} \left( \frac{R_{\eps}}{R} \right)^{\frac{4\alpha}{N-2}}\sum_{k=1}^m |\theta_k(R)|.$$
We fix $\alpha=m(N-2)$, and let $\rho \geq R_{\eps}$, so large that 
$$ C\,C_{\alpha}^{\frac{4}{N-2}} \rho^{m-1} \left( \frac{R_{\eps}}{\rho} \right)^{4m}\leq \frac{1}{2}.$$
Thus for all $R\geq \rho$,
$$\sum_{k=1}^m |\theta_k(2R)|\geq \frac{1}{2}\sum_{k=1}^m |\theta_k(R)|,$$
and by an easy induction
$$ \sum_{k=1}^m |\theta_k(2^nR)|\geq \frac{1}{2^n} \sum_{k=1}^m |\theta_k(R)|.$$
Using \eqref{MZ130} with $\alpha=m+10$, and letting $n\to\infty$, we obtain 
$$\sum_{k=1}^m |\theta_k(R)|=0$$
for all $R\geq \rho$. By \eqref{MZ20}, we deduce that $\|(u_0,u_1)\|_{\HHH_{\rho}}=0$, which concludes this step.
\end{step}
\begin{step}
 \label{St:k0m}
In this step, we assume that $k_0=m$ and that $\ell_m\neq 0$. Rescaling $u$ and replacing $u$ by $-u$ if necessary, we can assume $\ell_m=\left(N(N-2)\right)^{\frac N2-1}$. Since 
$$\left|\partial_r\left( W-\frac{\left(N(N-2)\right)^{\frac N2-1}}{r^{N-2}} \right)\right|\lesssim \frac{1}{r^{N+1}},$$
we deduce that for large $r$,
$$\left\|(W,0)- \ell_m \Xi_m\right\|_{\HHH(R)}\lesssim \sqrt{\int_R^{+\infty}\frac{r^{N-1}}{r^{2(N+1)}} \,dr}\approx\frac{1}{R^{m+\frac 32}}.$$
Let $h(t)=u(t)-W$, and $(h_0,h_1)=\vec{h}(0)$. By \eqref{asymptotic_NR} and the assumptions $k_0=m$, $\ell_m=\left(N(N-2)\right)^{\frac N2-1}$, 
noting that $\left( m-\frac 12 \right)\frac{N+2}{N-2}<m+\frac 12$, 
we have
\begin{equation}
\label{hsmall}                                                                                                                                             
\left\|(h_0,h_1)\right\|_{\HHH(R)}\lesssim \frac{1}{R^{m+\frac{1}{2}}}.                                                                                                                                            \end{equation} 
Furthermore, $h$ satisfies the following equation for $|x|>R+|t|$:
$$\partial_t^2h -\Delta h=F(h)+F(W+h)-F(W)-F(h).$$
Let 
$$\Gamma_R:=\left\{(t,x)\in \RR\times \RR^N\;:\; |x|>R+|t|\right\}.$$
By the fractional chain rule \eqref{fractional},
$$\|F(h)\|_{W'(\Gamma_R)}\lesssim \|h\|_{W(\Gamma_R)}\|h\|^{\frac{4}{N-2}}_{S(\Gamma_R}.$$
Furthermore,
$$\left|F(W+h)-F(W)-F(h)\right|\lesssim 
\begin{cases}
W^{\frac{4}{N-2}}|h|, & N\geq 7\\
 W^{\frac{4}{3}}|h|+ W |h|^{\frac 43},& N=5.
\end{cases} 
$$
By explicit computation, one has
$$\left\| \indic_{\Gamma_R}W^{\frac{4}{N-2}} \right\|_{L^{\frac{2(N+1)}{N+4}}_tL^{\frac{2(N+1)}{3}}_x} \lesssim \frac{1}{R^2}.$$
Thus if $N\geq 7$,
\begin{multline*}
 \left\|(F(W+h)-F(W)-F(h))\indic_{\Gamma_R}\right\|_{L^1_tL^2_x}
\lesssim \left\| W^{\frac{4}{N-2}} h \indic_{\Gamma_R}\right\|_{L^1_tL^2_x}\\
\lesssim \left\| \indic_{\Gamma_R}W^{\frac{4}{N-2}} \right\|_{L^{\frac{2(N+1)}{N+4}}_tL^{\frac{2(N+1)}{3}}_x} \left\| h \indic_{\Gamma_R}\right\|_{L^{\frac{2(N+1)}{N-2}}_{t,x}}\lesssim \frac{1}{R^2}\|h\|_{S(\Gamma_R)},
\end{multline*}
and by a similar computation, if $N=5$,
\begin{multline*}
  \left\|(F(W+h)-F(W)-F(h))\indic_{\Gamma_R}\right\|_{L^1_tL^2_x}
\\
\lesssim \left\| \indic_{\Gamma_R}W^{\frac{4}{3}} \right\|_{L^{\frac 43}_tL^{4}_x} 
\left\| h \indic_{\Gamma_R}\right\|_{L^{4}_{t,x}}+\|\indic_{\Gamma_R}W\|_{L^{\frac{7}{3}}_tL^{\frac{14}{3}}_x}\|h\|^{\frac{4}{3}}_{L^{\frac{7}{3}}_tL^{\frac{14}{3}}_x}\\ 
\lesssim \frac{1}{R^2}\|h\|_{S(\Gamma_R}+\frac{1}{R^{\frac 32}}\|h\|_{S(\Gamma_R)}^{\frac{4}{3}},
\end{multline*}
where the bound $\|\indic_{\Gamma_R}W\|_{L^{\frac{7}{3}}_tL^{\frac{14}{3}}_x}\lesssim \frac{1}{R^{\frac 32}}$ follows from the bound $W(r)\lesssim 1/r^3$ and explicit computations.
Using Strichartz estimates and the equation satisfied by $h$, we obtain
\begin{multline*}
\|h\|_{S(\Gamma_R)}+\|h\|_{W(\Gamma_R)}\\
\lesssim \|(h_0,h_1)\|_{\HHH(R)}+\|h\|_{W(\Gamma_R)}\|h\|^{\frac{4}{N-2}}_{S(\Gamma_R)}+\frac{1}{R^2}\|h\|_{S(\Gamma_R)}+\frac{1}{R^{3/2}}\|h\|^{\frac 43}_{S(\Gamma_R)}, 
\end{multline*}
where the last term is only necessary when $N=5$. This yields, taking $R$ large so that all the quantities appearing in this inequality are small:
\begin{equation}
\label{bndGammaR}
\|h\|_{S(\Gamma_R)}+\|h\|_{W(\Gamma_R)}\lesssim \|(h_0,h_1)\|_{\HHH(R)}. 
\end{equation} 
Let $h_F$ be the solution of the free wave equation with initial data $(h_0,h_1)$. Going back to the equation satisfied by $h$ and using again Strichartz estimates we obtain, in view of \eqref{bndGammaR},
$$\sup_{t\in \RR} \left\|\vec{h}(t)-\vec{h}_F(t)\right\|_{\HHH(R+|t|)}\lesssim \|(h_0,h_1)\|^{\frac{N+2}{N-2}}_{\HHH(R)}+\frac{1}{R^2}\|(h_0,h_1)\|_{\HHH(R)}.$$
In the case $N=5$, we have used the inequality $ab \leq \frac{3}{4}a^{\frac 43}+\frac{1}{4}b^{4}$ which implies
$$\frac{1}{R^{3/2}}\|(h_0,h_1)\|_{\HHH(R)}^{\frac 43}\lesssim \frac{1}{R^2}\|(h_0,h_1)\|_{\HHH(R)}+\|(h_0,h_1)\|^{\frac{7}{3}}_{\HHH(R)}.$$
Using the fact that $u$ (and thus $h$) is weakly nonradiative, we obtain by the exterior energy bound\eqref{linear_exterior},
\begin{multline*}
\|\pi_{P(R)}^{\bot}(h_0,h_1)\|_{\HHH(R)}\lesssim \sum_{\pm}\lim_{|t|\to\pm\infty} \sqrt{\int_{|x|>R+|t|} |\nabla_{t,x}h_F(t,x)|^2\,dx}\\
\lesssim \frac{1}{R^2}\|(h_0,h_1)\|_{\HHH(R)}+\|(h_0,h_1)\|^{\frac{N+2}{N-2}}_{\HHH(R)}, 
\end{multline*}
which yields
$$\|\pi_{P(R)^{\bot}}(h_0,h_1)\|_{\HHH(R)}\lesssim \frac{1}{R^2}\|\pi_{P(R)}(h_0,h_1)\|_{\HHH(R)}+\|\pi_{P(R)}(h_0,h_1)\|_{\HHH(R)}^{\frac{N+2}{N-2}}.$$
Using the bound \eqref{hsmall} on $\|(h_0,h_1)\|_{\HHH(R)}$, we see that we can drop the last term of this inequality:
\begin{equation}
\label{boundPRh}
\|\pi_{P(R)^{\bot}}(h_0,h_1)\|_{\HHH(R)}\lesssim \frac{1}{R^2}\|\pi_{P(R)}(h_0,h_1)\|_{\HHH(R)}. 
\end{equation} 
We denote by $(\eta_k(R))_{1\leq k\leq m}$ the coordinates of $\pi_{P(R)}(h_0,h_1)$ in the basis $(\Xi_k)_{1\leq k\leq m}$ of $P(R)$. Arguing exactly as in the proof of \eqref{MZ31} in Step \ref{St:asy_prelim}, we obtain for $1\ll R\leq R'\leq 2R$,
\begin{equation}
 \label{MZ31bis}
\sum_{k=1}^m\frac{|\eta_k(R)-\eta_k(R')|}{R^{k-\frac 12}}\lesssim \frac{1}{R^2}\sum_{k=1}^{m}\frac{1}{R^{k-\frac 12}}|\eta_k(R)|.
\end{equation} 
The end of the proof is very close to Step \ref{St:l0}. We first use \eqref{MZ31bis} and the bound 
$$\sum_{k=1}^{m}\frac{1}{R^{k-\frac 12}}|\eta_k(R)|\lesssim \frac{1}{R^{m+\frac{1}{2}}}$$
(that follows from \eqref{hsmall}) to prove 
\begin{equation}
\label{very_fast}
\forall \alpha\geq 1, \; \exists C_\alpha>0,\quad \sum_{k=1}^{m}\frac{1}{R^{k-\frac 12}}|\eta_k(R)|\leq \frac{C_{\alpha}}{R^{\alpha}}. 
\end{equation} 
Indeed, if 
$$\sum_{k=1}^{m}\frac{1}{R^{k-\frac 12}}|\eta_k(R)|\leq \frac{C_{\alpha}}{R^{\alpha}}$$
for some $\alpha \geq m$, then by \eqref{MZ31bis}, for $k\in \llbracket 1,m\rrbracket$, large $R$ and $n\geq 0$ we have
$$ |\eta_k(2^nR)-\eta_k(2^{n+1}R)|\leq C_{\alpha}(2^nR)^{k-\alpha-5/2},$$
which yields, summing up over $n$, that there exists $C_{\alpha}'>0$ such that
$$\sum_{k=1}^{m}\frac{1}{R^{k-\frac 12}}|\eta_k(R)|\leq \frac{C_{\alpha}'}{R^{\alpha+2}}.$$
This proves \eqref{very_fast}. On the other hand, \eqref{MZ31bis} implies, for large $R$, 
$$ \sum_{k=1}^m \frac{1}{(2R)^{k-\frac 12}}|\eta_k(2R)|\geq \frac{1}{2^m}\sum_{k=1}^m \frac{1}{R^{k-\frac 12}}|\eta_k(R)|.$$
By an elementary induction, we obtain that for large $R$, and any natural integer $n$,
$$\sum_{k=1}^m \frac{1}{(2^nR)^{k-\frac 12}}|\eta_k(2^nR)|\geq \frac{1}{2^{nm}}\sum_{k=1}^m \frac{1}{R^{k-\frac 12}}|\eta_k(R)|,$$
which proves by \eqref{very_fast}, letting $R\to\infty$, that $\eta_k(R)=0$ for all $k\in \llbracket 1,m\rrbracket$. Combining with \eqref{boundPRh} we deduce that $(h_0,h_1)(r)=0$ a.e. for large $r$, concluding the proof.
\end{step}
 \end{proof}
\begin{proof}[Proof of Proposition \ref{P:l_no_t}]
 
 If $u_F$ is a solution of the free wave equation, one can deduce from finite speed of propagation and energy conservation
 $$ \forall R>0,\; \forall T\in \RR,\quad \|\vec{u}_F(T)\|_{\HHH(R+|T|)}\leq \|\vec{u}_{F}(0)\|_{\HHH(R)}.$$
 Let $u$ be a weakly nonradiative solution, and $\ell(T)$, $k_0(T)$ defined by Proposition \ref{P:asymptotic_NR}. We can assume $k_0(T)\leq m-1$, since in the case $k_0(T)=m$, by Proposition \ref{P:asymptotic_NR}, the solution $u$ is independent of $t$ on $|x|\geq R+|T|$ for some large $R$.

 We fix $T\in \RR$. By Proposition \ref{P:asymptotic_NR}, we have
 \begin{multline}
 \label{norm_uT}
\|\vec{u}(T)\|_{\HHH(R)}=\ell(T)\|\Xi_{k_0(T)}\|_{\HHH(R)}+o\left( R^{-k_0(T)+\frac 12} \right)\\
= c_{k_0(T)}\ell(T) R^{-k_0(T)+\frac 12}+o\left( R^{-k_0(T)+\frac 12} \right)  
 \end{multline}
as $R\to\infty$, where the constant $c_{k_0(T)}$ is defined in \eqref{normXi}. Furthermore by the small data well-posedness theory, for large $R$,
\begin{multline*}
 \|\vec{u}(T)\|_{\HHH(R+|T|)}\leq \|u_F(T)\|_{\HHH(R+|T|)}+C\|(u_0,u_1)\|_{\HHH(R)}^{\frac{N+2}{N-2}}\\
 \leq \|(u_0,u_1)\|_{\HHH(R)}+C\|(u_0,u_1)\|_{\HHH(R)}^{\frac{N+2}{N-2}}.
\end{multline*}
Combinining with \eqref{norm_uT}, we deduce that for any $T\in \RR$, 
\begin{multline*}
c_{k_0(T)}\ell(T) (R+|T|)^{-k_0(T)+\frac 12}+o\left((R+|T|)^{-k_0(T)+\frac 12}  \right)\\
\leq c_{k_0(0)}\ell(0) R^{-k_0(0)+\frac 12}+o\left(R^{-k_0(0)+\frac 12}\right),\quad R\to \infty.
\end{multline*}
Letting $R\to \infty$, we see that 
$$k_0(T)\geq k_0(0)$$
and 
$$k_0(T)=k_0(0)\Longrightarrow \ell(T)\leq \ell(0).$$
Using the same argument on $(t,x)\mapsto u(T-t,x)$, which is also a weakly radiative solution of \eqref{NLW}, we deduce
$$k_0(0)\geq k_0(T)$$
and 
$$k_0(T)=k_0(0)\Longrightarrow \ell(0)\leq \ell(T).$$
Combining, we obtain as announced, 
$$ k_0(0)=k_0(T)\text{ and }\ell(0)=\ell(T).$$
\end{proof}
\section{Compactly supported initial data}
\label{S:compact}
In this section we prove Proposition \ref{P:nonradi_compact} and Theorem \ref{T:support}. Let $(u_0,u_1)\in \HHH_{\rad}\setminus\{(0,0)\}$ with compact support. Let
Let $u$ be the solution of \eqref{NLW} (or \eqref{LW}) with initial data $(u_0,u_1)$, and 
$$\rho(t)=\min\left\{ \rho\;:\; \int_{\rho}^{+\infty} (\partial_{t,r}u)^2r^{N-1}\,dr=0\right\},\quad \rho(t)=\rho_0.$$
By finite speed of propagation:
$$\forall t\in I_{\max}(u),\quad \rho(t)\leq \rho_0+|t|.$$
By the small data theory, if $\eps$ is small enough, then $u$ is well-defined in $\{r>\rho_0-\eps+|t|\}$ and
\begin{equation}
\label{bound_norm}
\sup_{t\in \RR}\|\vec{u}(t)\|_{\HHH(\rho_0-\eps+|t|)}<\infty. 
\end{equation} 
Furthermore, $u$ satisfies,
$\partial_t^2u-\Delta u=V u$,
where $V(t,r)=|u|^{\frac{4}{N-2}}$, and it follows from \eqref{bound_norm} and the radial Sobolev inequality that
$$r\geq \rho_0-\eps+|t|\Longrightarrow |V(t,r)|\lesssim \frac{C(u)}{r^2},$$
for some constant $C(u)$ depending on $u$. By Proposition \ref{P:support}, if $\rho_0-\eps<R<\rho_0$ (taking a smaller $\eps$ if necessary), the following holds for all $t\geq 0$ or for all $t\leq 0$:
\begin{equation*}
 \int_{R+|t|}^{+\infty} (\partial_{t,r}u(t,r))^2r^{N-1}\,dr\geq \frac{1}{8}\int_{R}^{+\infty} (\partial_{t,r}u(0,r))^2r^{N-1}\,dr>0.
\end{equation*} 
This concludes the proof of Proposition \ref{P:nonradi_compact}, since the preceding lower bound is independent of $t$. We see also that this lower bound imply $\rho(t)\geq R+|t|$, whenever $R<\rho_0$. This yields the conclusion of Theorem \ref{T:support}.

\appendix

\section{Sequences with geometric growth}
\label{A:sequences}
In this appendix we prove Claim \ref{Cl:sequences}. In all the proof $C$ (respectively $\eps$) will denote a large (respectively small) constant, that may change from line to line and is allowed to depend on $r$, $q$, $c_0$ and $\beta$, but not on the other parameters.

We first assume $c_0=0$. Thus we have
$$ \forall n\geq 0,\quad \mu_{n+1}\leq q\mu_n+\nu_0r^n.$$
By a straightforward induction, we obtain
\begin{equation}
\label{geometric}
\forall n\geq 0,\quad \mu_{n}\leq q^n\mu_0+\nu_0 r^{n-1}\sum_{j=0}^{n-1} \left( \frac{q}{r} \right)^{j}. 
\end{equation} 
If $q\neq r$ we deduce
$$\mu_n\leq q^n \mu_0+\nu_0r^{n-1}\frac{1-\left( \frac{q}{r} \right)^n}{1-\frac{q}{r}}.$$
In the case where $q<r$, this yields
$$\mu_n\leq q^n\mu_0+\frac{\nu_0}{1-\frac{q}{r}} r^{n-1}\leq C(\mu_0+\nu_0)r^n.$$
When $q>r$, we have 
$$\mu_n\leq q^n\mu_0+\nu_0r^{n-1}\frac{\left( \frac{q}{r} \right)^n}{\frac{q}{r}-1}\leq C(\mu_0+\nu_0)q^n.$$
Finally, in the case $q=r$, the inequality \eqref{geometric} is
$$ \mu_n\leq r^n\mu_0+\nu_0 n r^n\leq C(\mu_0q^n+n\nu_0r^n).$$

We next treat the general case. We first note that the assumptions \eqref{AZ10} and \eqref{AZ11} imply
\begin{equation}
 \label{AZ20}
 \mu_{n+1} \leq \left(q+c_0\eps^{\beta-1}\right)\mu_n+\nu_0r^n.
\end{equation} 
If $q<r$, we choose $\eps$ so small, so that $(q+c_0\eps^{\beta-1})\leq \frac{q+r}{2}<r$. Using the case $c_0=0$ treated previously, we obtain
$$\mu_k\leq C(\mu_0+\nu_0)r^k.$$
If $q\geq r$, we have $q+c_0\eps^{\beta-1}>r$, and \eqref{AZ20} implies, using the case $c_0=0$, 
\begin{equation}
 \label{AZ21}
 \mu_n\leq C \left(q+c_0\eps^{\beta-1}\right)^{n}(\mu_0+\nu_0).
\end{equation} 
Plugging this into \eqref{AZ10} we deduce,
$$\mu_{n+1}\leq q\mu_n+\nu_0r^n+ C^{\beta}(\mu_0+\nu_0)^{\beta}c_0\left( q+c_0\eps^{\beta-1} \right)^{n\beta}.$$
We choose $\eps$ small, so that $q'=(q+c_0\eps^{\beta-1})^{\beta}<q$, which is possible since $\beta>1$ and $q<1$.  We deduce
$$\mu_{n+1}\leq q\mu_n+C (\nu_0+\mu_0) (\max(q',r))^n.$$
We use again the case $c_0=0$. 
If $r<q$, we obtain \eqref{AZ12}. If $q=r$, we have $\max(q',r)=r=q$ and \eqref{AZ13} follows. \qed

\bibliographystyle{acm}
\bibliography{toto}
\end{document}